\documentclass[11pt,amssymb]{amsart}
\usepackage{amssymb}
\usepackage[mathscr]{eucal}
\usepackage[all,cmtip]{xy}

\newcommand{\Z}{{\mathbb Z}}

\newcommand{\Q}{{\mathbb Q}}

\newcommand{\R}{{\mathbb R}}

\newcommand{\Br}{\mathrm{Br}}

\newcommand{\Ga}{\mathrm{Gal}}
\newtheorem{thm}{Theorem}[section]
\newtheorem{lemma}[thm]{Lemma}
\newtheorem{prop}[thm]{Proposition}
\newtheorem{cor}[thm]{Corollary}

\newcommand{\gen}{\mathbf{gen}}

\newcommand{\Div}{\mathrm{Div}}
\newcommand{\Pic}{\mathrm{Pic}}

.240pk scaled 1200 .240pk
\begin{document}

\title[The finiteness of genus]{The finiteness of the genus of a finite-dimensional division algebra, and some generalizations}

\author[V.~Chernousov]{Vladimir I. Chernousov}
\author[A.~Rapinchuk]{Andrei S. Rapinchuk}
\author[I.~Rapinchuk]{Igor A. Rapinchuk}

\begin{abstract}
We prove that the genus of a finite-dimensional division algebra is finite whenever the center is a finitely generated field of any characteristic.
We also discuss potential applications of our method to other problems, including the finiteness of the genus of simple algebraic groups of type $\textsf{G}_2$. These applications involve the double cosets of adele groups of algebraic groups over arbitrary finitely generated fields: while over number fields these double cosets are associated with the class numbers of algebraic groups and hence have been actively analyzed, similar questions over more general fields seem to come up for the first time. In the Appendix, we link the double cosets with $\check{\rm C}$ech cohomology and indicate connections between certain finiteness properties involving double cosets (Condition (T)) and Bass's finiteness conjecture in $K$-theory.
\end{abstract}

\address{Department of Mathematics, University of Alberta, Edmonton, Alberta T6G 2G1, Canada}

\email{vladimir@ualberta.ca}

\address{Department of Mathematics, University of Virginia,
Charlottesville, VA 22904-4137, USA}

\email{asr3x@virginia.edu}

\address{Department of Mathematics, Michigan State University, East Lansing, MI
48824, USA}

\email{rapinch@math.msu.edu}

\maketitle

\vskip3mm

\centerline{\it\large To Louis Rowen on the occasion of his retirement}

\vskip5mm

\section{Introduction}\label{S:Intro}

Let $D$ be a finite-dimensional central division algebra over a field $K$. The genus $\gen(D)$ is defined to be the set of classes $[D']$ in the Brauer group $\Br(K)$ represented by central division $K$-algebras $D'$ that have the same maximal subfields as $D$ (the latter means that $D$ and $D'$ have the same degree $n$, and a degree $n$ extension $P/K$ admits a $K$-embedding $P \hookrightarrow D$ if and only if it admits a $K$-embedding $P \hookrightarrow D'$)\footnote{We refer the reader to \cite[\S5]{CRR1a} for different variations of the notion of the genus. The referee also suggested to define the genus of an element of the Brauer group as the genus of the underlying division algebra. While the latter notion may be useful, some caution needs to be exercised in interpreting its consequences: for example, we don't know whether the fact that the division $D$ and $D'$ have the same maximal subfields implies that the matrix algebras $\mathrm{M}_{\ell}(D)$ and $\mathrm{M}_{\ell}(D')$ have the same maximal subfield / \'etale subalgebras for any (or even some) $\ell > 1$. (On the other hand, it is known that if the matrix algebras $\mathrm{M}_{\ell}(D)$ and $\mathrm{M}_{\ell'}(D')$, where $D$ and $D'$ are central division algebras over $K$, have the same maximal \'etale subalgebras, then $\ell = \ell'$ and $D$ and $D'$ have the same maximal separable subfields, cf. \cite[Lemma 2.3]{RR}.)}. One of the goals of the current paper is to give a simple proof of the following
\begin{thm}\label{T:Finite1}
Let $K$ be a finitely generated field. For any finite-dimensional central division algebra $D$ over $K$, the genus $\gen(D)$ is finite.
\end{thm}

Earlier, we proved this fact in the case where the degree $n$ is prime to $\mathrm{char}\: K$ -- see \cite{CRR1}, \cite{CRR2}. The argument is based on the analysis of ramification with respect to valuations in a suitable set $V$ of discrete valuations of $K$, and consists of two parts. First, we proved that if $V$ satisfies some natural conditions, then we have the following estimate on the size of the genus:
\begin{equation}\tag{E}\label{E:estimate}
\vert \gen(D) \vert \leqslant \vert {}_n\Br(K)_V \vert \cdot \varphi(n)^r,
\end{equation}
where ${}_n\Br(K)_V$ consists of those elements in the $n$-torsion subgroup ${}_n\Br(K)$ that are unramified at all $v \in V$, $r$ is the number of those $v \in V$ where $D$ is ramified, and $\varphi$ is Euler's function (see \cite[Theorem 2.2]{CRR1}). We then showed that when $n$ is prime to the characteristic of a finitely generated field $K$, the field can be equipped with a set $V$ of discrete valuations for which the unramified Brauer group ${}_n\Br(K)_V$ is finite, of order that can be explicitly estimated in some cases (see \cite{CRR2}). This argument not only gives the finiteness of $\gen(D)$ but also provides a uniform bound on its size over all division algebras $D$ of a given degree $n$ and a given number $r$ of ramified places. In the present paper, we will show that the mere finiteness of $\gen(D)$ can be established by a simpler argument which doesn't give any estimates but on the other hand, works also in the case where $n$ is divisible by $\mathrm{char}\: K$. (We note that the assumption that $K$ is finitely generated is essential as the genus of a division algebra over an infinitely generated field can be infinite, cf. \cite{Meyer}, \cite{Tikhon}.)

The proof of Theorem \ref{T:Finite1} also uses the analysis of ramification, but bypasses some technical arguments developed in \cite{CRR1}, \cite{CRR2}. In particular, instead of the estimate (\ref{E:estimate}) we need the fact that (under some minor assumptions) if $D$ is unramified at all $v \in V$, then any $D'$ with the same maximal subfields retains this property - see Lemma \ref{L:unram}. The rest of the argument uses the following considerations. As we already described in \cite{CRR2}, to a given set $V$ of discrete valuations of a field $K$ satisfying a natural condition (see condition (A) in \S\ref{S:F1}), one can associate two groups: the Picard group $\Pic(K , V)$ and the group of units $\mathrm{U}(K , V)$. Our argument essentially shows that if an arbitrary field $K$ can be equipped with a set $V$ of discrete valuations such that the groups $\Pic(L , V^L)$ and $\mathrm{U}(L , V^L)$ are finitely generated for all finite separable  extensions $L/K$, where $V^L$ consists of all extensions of places from $V$ to $L$, then $\gen(D)$ is finite for any finite-dimensional division $K$-algebra $D$ (if $D$ is a quaternion algebra then it  suffices to require the finite generation of $\mathrm{U}(L , V^L)$ and $\Pic(L , V^L)$ for a single separable quadratic subfield $L$ of $D$).  On the other hand, a  set of discrete valuations $V$ with this property is available for any finitely generated field - one can take a ``divisorial" set of places for a model of $K$ of finite type. We note that the argument gives the finiteness of the unramified relative Brauer group $\Br(L/K)_V$ for any fixed  separable extension $L/K$ which is either cyclic or has a prime power degree -- while this is sufficient to establish the finiteness of $\gen(D)$ for a division $K$-algebra $D$ of any degree $n$, it does not give the finiteness of ${}_n\Br(K)_V$.

In \S\ref{S:G2}, we discuss how the method of proof of Theorem \ref{T:Finite1} can be potentially extended to other situations, in particular, to prove the finiteness of the genus $\gen_K(G)$ of simple $K$-groups of type $\textsf{G}_2$. For this purpose, the assumption that $\Pic(K , V)$ be finitely generated needs to be upgraded. What we actually use in the proof of Theorem \ref{T:Finite1} is the following consequence of this assumption: there exists a subset $V' \subset V$ with a finite complement $V \setminus V'$ such that $\Pic(K , V')$ is trivial. On the other hand, it is well-known that $\Pic(K , V)$ can be described in terms of ideles, and that this idelic description readily lends itself to a generalization in terms of double cosets of the adele groups of algebraic groups. So, along these lines, we formulate a condition, termed Condition (T), for an arbitrary algebraic $K$-group $G$, and then show that its truth for the group $G = \mathrm{GL}_{1 , D}$ (or even $\mathrm{GL}_{\ell , D}$ with $\ell \geqslant 1$) where $D$ is a quaternion algebra, implies the finiteness of some parts of the unramified cohomology group $H^3(K , \mu_2)_V$ where $\mu_2 = \{ \pm 1 \}$ (we recall that  so-called decomposable elements of $H^3(K , \mu_2)$ classify the $K$-forms of type $\textsf{G}_2$). It should be noted that while the double cosets of adele groups of algebraic groups over number fields have been studied rather extensively (cf. \cite[Ch. VIII]{PlRa}), no results over more general fields seem to be available in the literature. So, in \S\ref{S:T} we confirm (T) for split groups over the function fields of curves over finitely generated fields in the situation where $V$ is the set of all geometric places (Theorem \ref{T:CondT}) and for arbitrary tori over finitely generated fields with respect to divisorial $V$ (Proposition \ref{P:CondT}).

The Appendix that comprises \S\S5-7 was written after the main body of the paper had been accepted for publication. It provides an account of the variety of links between the analysis of double cosets and other issues. In particular, in \S5 we show that the map from $\check{\rm C}$ech cohomology to the set of double cosets constructed by G.~Harder \cite{Harder} for  group schemes over Dedekind rings exists in fact in the general situation where it is always injective but,  unlike in the Dedekind case, may not be surjective. In \S6 we describe this map explicitly for the group $\mathrm{GL}_n$ in terms of projective and reflexive modules and show that in some situations our Condition (T) for this group can be derived from Bass's finiteness conjecture in $K$-theory
(cf. \cite{Bass-Conj}). Finally, in \S7 we develop the descent procedure that, under certain additional assumptions, enables one to derive Condition (T) for $\mathrm{GL}_{\ell , D}$, where $D$ is a finite-dimensional central  division algebra, from its truth for $\mathrm{GL}_n$.

\section{Proof of Theorem \ref{T:Finite1}}\label{S:F1}

\noindent {\bf 1. Picard and unit groups.} Let $K$ be a field, and let $V$ be a set of discrete valuations of $K$ that satisfies the following condition

\vskip2mm

(A) For any $a \in K^{\times}$, the set $V(a) := \{ v \in V \, \vert \, v(a) \neq 0 \}$ is finite.

\vskip2mm

\noindent We let $\Div(K , V)$ denote the free abelian group on the set $V$, the elements of which will be called ``{\it divisors.}'' The fact that $V$ satisfies (A) enables one to associate to any $a \in K^{\times}$ the corresponding ``{\it principal divisor}''
$$
(a) = \sum_{v \in V} v(a) \cdot v.
$$
Let $\mathrm{P}(K , V)$ denote the subgroup formed by all principal divisors. We will call the quotient $$\Div(K , V)/\mathrm{P}(K , V)$$ the {\it Picard group} of $V$ and denote it by $\Pic(K , V)$. We note that if $\Pic(K , V)$ is finitely generated, then there exists a finite subset $S \subset V$ such that $\Pic(K , V \setminus S) = 0$. Indeed, let  $d_1, \ldots , d_r \in \Div(K , V)$ be divisors whose images in $\Pic(K , V)$ generate this group. We can find a finite subset $S \subset V$ such that each $d_i$ lies in the subgroup $\langle S \rangle \subset \Div(K , V)$ generated by $S$.
Then
$$
\Div(K , V) = \langle S \rangle + \mathrm{P}(K , V).
$$
Let $\pi_S \colon \Div(K , V) \to \Div(K , V \setminus S)$ be the canonical projection defined by $\pi_S(v) = v $ for $v \in V \setminus S$ and $\pi(v) = 0$ for $v \in S$. Applying $\pi_S$ to the above equality and observing that $\pi_S(\mathrm{P}(K , V)) = \mathrm{P}(K , V \setminus S)$, we obtain that $\Div(K , V \setminus S) = \mathrm{P}(K , V \setminus S)$, i.e. $\Pic(K , V \setminus S) = 0$, as required.

We will also need the corresponding ``{\it group of units}"
$$
\mathrm{U}(K , V) = \{ a \in K^{\times} \ \vert \ v(a) = 0 \ \ \text{for all} \ \ v \in V \}.
$$
We note that for any $v \in V$, we have the following exact sequence
$$
\mathrm{U}(K , V) \longrightarrow \mathrm{U}(K , V \setminus \{ v \}) \stackrel{v}{\longrightarrow} \Z.
$$
It follows that if $\mathrm{U}(K , V)$ is finitely generated, then $\mathrm{U}(K , V \setminus \{ v \})$ is also finitely generated. Iterating, we obtain that if $\mathrm{U}(K , V)$ is finitely generated, then for any finite subset $S \subset V$, the group $\mathrm{U}(K , V \setminus S)$ is finitely generated. Thus, if both groups $\Pic(K , V)$ and $\mathrm{U}(K , V)$ are finitely generated, we can find a finite subset $S \subset V$ such that $\Pic(K , V \setminus S) = 0$ and then $\mathrm{U}(K , V \setminus S)$ will still be finitely generated.

\vskip.5mm

Suppose now that $K$ is a field that is finitely generated over its prime subfield.
We can then find an integrally closed finitely generated $\Z$-subalgebra $R \subset K$ such that $K$ is the  field of fractions of $R$. In this case, the scheme $X = \mathrm{Spec}\: R$ is regular in codimension 1 (cf. \cite[Ch. II, \S 6]{Hart}), and consequently  to every closed irreducible subscheme $Y \subset X$ of codimension 1 (equivalently, to any height one prime ideal $\mathfrak{p} \subset R$) there corresponds a discrete valuation $v_Y$ (or $v_{\mathfrak{p}}$) of $K$, and we call the set
$$
V^K = \{ v_Y \ \vert \ Y \subset X \ \ \text{closed, irreducible}, \ \ \mathrm{codim}_X\: Y = 1 \}
$$
a {\it divisorial} set of places of $K$. We note that the divisorial sets of places associated with different choices of $R$ differ in finitely many places. We have
$$
R = \bigcap_{v \in V^K} \mathscr{O}_{K , v},
$$
where $\mathscr{O}_{K , v}$ is the valuation ring of $v$ in $K$ (cf. \cite[Ch. II, Proposition 6.3A]{Hart}). It follows that the corresponding unit group $\mathrm{U}(K , V^K)$ coincides with $R^{\times}$, hence is finitely generated (see \cite{Samuel}). Furthermore, the group $\Pic(K , V^K)$ is identified with the class group $\mathrm{Cl}(R)$, which is also finitely generated (see \cite{Kahn}).

Moreover, let $L/K$ be a finite separable extension. Then the integral closure $R'$ of $R$ in $L$ is an $R$-module of finite type, hence an integrally closed finitely generated $\Z$-algebra with field of fractions $L$. The corresponding divisorial set of valuations $V^L$ coincides with the set of all extensions of places in $V^K$ to $L$. It follows from the above discussion that the groups $\Pic(L , V^L)$ and $\mathrm{U}(L , V^L)$ are finitely generated.

\vskip2mm

\noindent {\bf 2. Reduction to division algebras of a prime power degree.}  We begin with the following.

\begin{lemma}\label{L:primary}
Let $D$ be a central division algebra of degree $n$ over a field $K$. Assume that $n = n_1n_2$, with $n_1$ and $n_2$ relatively prime, and  $D = D_1 \otimes_K D_2$, where $D_i$ is a central division $K$-algebra of degree $n_i$ for $i = 1, 2$.   Then any central division $K$-algebra $D'$ of degree $n$
that has the same maximal subfields as $D$ is of the form $D' = D'_1 \otimes_K D'_2$ where $[D'_i] \in \gen(D_i)$ for $i = 1, 2$.
\end{lemma}
\begin{proof}
We can write $D' = D'_1 \otimes_K D'_2$ where $D'_i$ has degree $n_i$ for $i = 1, 2$ (cf. \cite[Part I, Theorem 4.19]{DeFa}). To show that $[D'_i] \in \gen(D_i)$, take arbitrary maximal subfields $F_i$ of $D_i$ for $i = 1, 2$. Then $F := F_1 \otimes_K F_2$ is a maximal subfield of $D$, hence admits an embedding $F \hookrightarrow D'$. So,
$$
D' \otimes_K F = \mathrm{M}_n(F) = (D'_1 \otimes_K F) \otimes_F (D'_2 \otimes_K F).
$$
Since $n_i [D'_i \otimes_K F] = 0$ in $\Br(F)$, and $n_1$ and $n_2$ are relatively prime, we conclude that $[D'_i \otimes_K F] = 0$ for both $i = 1, 2$. Furthermore,
$$
D'_i \otimes_K F = (D'_i \otimes_K F_i) \otimes_{F_i} F,
$$
and since $[F : F_i] = n_{3-i}$ is relatively prime to $n_i$, we conclude that the class $[D'_i \otimes_K F_i]$ is trivial in $\Br(F_i)$, and consequently there are $K$-embedding $F_i \hookrightarrow D'_i$ for $i = 1, 2$. Conversely, starting with arbitrary maximal subfields $F'_i$ of $D'_i$ $(i = 1, 2)$ and repeating the same argument, one shows that there are $K$-embeddings $F'_i \hookrightarrow D_i$. So, our claim follows.
\end{proof}

\noindent {\bf Remark 2.2.} The referee asked an interesting (and apparently difficult) question of whether the converse is also true, viz. for any $D'_i \in \gen(D_i)$ in the above notations, the division algebra $D' = D'_1 \otimes_K D'_2$ lies in $\gen(D)$.

\addtocounter{thm}{1}

\vskip1mm

We can now show that it suffices to prove Theorem \ref{T:Finite1} for division algebras of a prime power degree.
\begin{cor}\label{C:primary}
Let $K$ be a field. If $\gen(D)$ is finite for any central division $K$-algebra whose degree is a prime power, then $\gen(D)$ is finite for any central
division $K$-algebra $D$.
\end{cor}

Indeed, let $D$ be a central division $K$-algebra of degree $n = p_1^{\alpha_1} \cdots p_r^{\alpha_r}$. Then we can write
$$
D = D_1 \otimes_K \cdots \otimes_K D_r
$$
where $D_i$ is a central division $K$-algebra of degree $p_i^{\alpha_i}$. It follows from Lemma \ref{L:primary} that any central division $K$-algebra $D'$ of degree $n$ that has the same maximal subfields as $D$ is of the form
$$
D' = D'_1 \otimes_K \cdots \otimes_K D'_r \ \ \text{where} \ \ [D'_i] \in \gen(D_i),
$$
and our assertion follows.

\vskip2mm

\noindent {\bf 3. Unramified relative Brauer groups for a cyclic extension.} The subsequent steps in the proof of Theorem \ref{T:Finite1} involve the analysis of ramification, so we briefly recall the necessary facts. Let $\mathscr{K}$ be a field complete with respect to a discrete valuation $v$. Given a (finite-dimensional) central division $\mathscr{K}$-algebra $\mathscr{D}$, the valuation $v$ extends uniquely to a valuation $w$ of $\mathscr{D}$. Then $\mathscr{D}$ is {\it unramified} if the ramification index $e(w \vert v) = [w(\mathscr{D}^{\times}) : v(\mathscr{K}^{\times})]$ is $1$. An unramified algebra necessarily splits over the maximal unramified extension $\mathscr{K}_{\mathrm{ur}}$. Conversely, if $\mathscr{D}$ splits over $\mathscr{K}_{\mathrm{ur}}$ then the property of $\mathscr{D}$ being unramified can be characterized by several equivalent conditions: (1) the center of the residue algebra $D^{(w)}$ coincides with the residue field $\mathscr{K}^{(v)}$; (2) there exists an Azumaya algebra $\mathscr{A}$ over the valuation ring $\mathscr{O} \subset \mathscr{K}$ such that $\mathscr{D} = \mathscr{A} \otimes_{\mathscr{O}} \mathscr{K}$; (3) if $r \colon H^2(\mathscr{K}_{\mathrm{ur}}/\mathscr{K} , \mathscr{K}_{\mathrm{ur}}^{\times}) \to H^1(\mathscr{K}^{(v)} , \Q/\Z)$ is the residue map, then $r([\mathscr{D}])$ is trivial (see \cite{Wads}). More generally, a central simple $\mathscr{K}$-algebra $\mathscr{C}$ is unramified if $\mathscr{C} = \mathrm{M}_{\ell}(\mathscr{D})$ where $\mathscr{D}$ is an unramified division algebra. We will need the (well-known)  explicit computation of the residue map $r$ on cyclic algebras.

Let $L/K$ be a degree $n$ cyclic extension of arbitrary fields, and let $\sigma$ be a generator of the Galois group $\Ga(L/K)$. Given $c \in K^{\times}$, we let   $(L, \sigma, c)$ denote the cyclic algebra
$$
C = L \oplus Lx \oplus \cdots \oplus Lx^{n-1}
$$
with the relations $xax^{-1} = \sigma(a)$ for $a \in L$ and $x^n = c$.
\begin{lemma}\label{L:cycl1}
Let $\mathscr{K}$ be a field complete with respect to a discrete valuation $v$, let $\mathscr{L}/\mathscr{K}$ be an unramified cyclic extension of degree $m$,
and let $\tau$ be a generator of $\Ga(\mathscr{L}/\mathscr{K})$. Then for the residue map $r \colon H^2(\mathscr{K}_{\mathrm{ur}}/\mathscr{K} , \mathscr{K}_{\mathrm{ur}}^{\times}) \to H^1(\mathscr{K}^{(v)} , \Q/\Z)$ and $c \in \mathscr{K}^{\times}$, the image $r([(\mathscr{L}, \tau, c)])$ corresponds to the character $\chi \colon \Ga(\mathscr{L}^{(w)}/\mathscr{K}^{(v)}) \to \Q/\Z$ such that $\chi(\bar{\tau}) = v(c)/m$, where $\bar{\tau}$ is the image of $\tau$ under the natural identification $\Ga(\mathscr{L}/\mathscr{K}) \simeq \Ga(\mathscr{L}^{(w)}/\mathscr{K}^{(v)})$ and $w$ is the extension of $v$.
\end{lemma}
\begin{proof}
Recall (cf. \cite[Ch. XIII, \S3]{Serre-LF}) that $r$ is the following composition
$$
H^2(\mathscr{K}_{\mathrm{ur}}/\mathscr{K} , \mathscr{K}_{\mathrm{ur}}^{\times}) \longrightarrow H^2(\mathscr{K}_{\mathrm{ur}}/\mathscr{K} , \Z) \longrightarrow
H^1(\mathscr{K}_{\mathrm{ur}}/\mathscr{K} , \Q/\Z) \simeq H^1(\mathscr{K}^{(v)} , \Q/\Z),
$$
where the first map is induced by the valuation $\mathscr{K}_{\mathrm{ur}}^{\times} \to \Z$ extending $v$ and the second one is the inverse of the connecting map corresponding to the exact sequence
$$
0 \to \Z \longrightarrow \Q \longrightarrow \Q/\Z \to 0
$$
of trivial Galois modules. The cocycle in $H^2(\mathscr{L}/\mathscr{K} , \mathscr{L}^{\times})$ that corresponds to $[(\mathscr{L}, \tau, c)]$ is given by
$$
h(\tau^i , \tau^j) = \left\{  \begin{array}{ccl} 1 & , & i+j < m, \\ c & , & \text{otherwise}, \end{array} \right. \ \ \text{for} \ \ 0 \leqslant i , j < m.
$$
The image of this cocycle in $H^2(\mathscr{L}/\mathscr{K} , \Z)$ is given by a similar formula where $c$ is replaced by $v(c)$. Using now the description of the connecting map, we obtain that it takes this cocycle to the character in $H^1(\mathscr{L}/\mathscr{K} , \Q/\Z)$ given by
$$
\chi(\tau) = \frac{v(c)}{m},
$$
and our claim follows using inflation maps.
\end{proof}

Now, a central simple algebra $C$ over an arbitrary field $K$ is said to be unramified at a discrete valuation $v$ of $K$ if the corresponding algebra $C_v := C \otimes_K K_v$ over the completion $K_v$ is unramified as defined above. Given a set $V$ of discrete valuations of $K$, we let $\Br(K)_V$ denote the subgroup of $\Br(K)$ consisting of classes that are unramified at all $v \in V$. Finally, for an extension $L/K$, we let $\Br(L/K)$ denote the corresponding relative Brauer group, i.e.
$$
\Br(L/K) = \mathrm{Ker}\left(\Br(K) \to \Br(L)\right),
$$
and set $\Br(L/K)_V = \Br(L/K) \cap \Br(K)_V$.

\begin{prop}\label{P:cycl}
Let $K$ be a field equipped with a set $V^K$ of discrete valuations that satisfies condition {\rm (A)}, let $L/K$ be a cyclic extension of degree $n$, and let $V^L$ be the set of all extensions of places from $V^K$ to $L$.\footnotemark Assume that the groups $\Pic(L , V^L)$ and $\mathrm{U}(K , V^K)$ are finitely generated. Then the unramified relative Brauer group $\Br(L / K)_{V^K}$ is finite.
\end{prop}
\footnotetext{We note that if $V^K$ satisfies condition (A), then so does the set $V^F$ of all extensions of places from $V^K$ to any finite extension $F/K$.}
\begin{proof}
Suppose $L = K(a)$, and let $f(t) \in K[t]$ be the minimal polynomial of $a$ over $K$. For all but finitely many $v \in V^K$, we have $f \in \mathscr{O}_{K , v}[t]$ and the reduction of $f$ modulo the valuation ideal does not have multiple roots; then $v$ is unramified in $L$. So, deleting from $V^K$ a finite set of valuations, we may assume that all $v \in V^K$ are unramified in $L$ (note that this does not affect finite generation of $\Pic(L , V^L)$ and $\mathrm{U}(K , V^K)$). Furthermore, deleting from $V^K$ another finite set of valuations, we may assume in addition that $\Pic(L , V^L) = 0$, while $U(K , V^K)$ is still finitely generated.

Now, fix a generator $\sigma$ of $\Ga(L/K)$. Then any element of $\Br(L/K)$ is represented by some $C = (L, \sigma, c)$, $c \in K^{\times}$. For $v \in V^K$, pick an extension $w$ to $L$ and set $n_v = [L_w : K_v]$. It is easy to see that $[C_v] \in \Br(K_v)$ is represented by $(L_w, \sigma^{d_v}, c)$ where $d_v = n/n_v$, so it follows from Lemma \ref{L:cycl1} that $C$ is unramified at $v$ if and only if
$$
v(c) \equiv 0(\mathrm{mod}\: n_v).
$$
To continue the argument, we need the following.
\begin{lemma}\label{L:cycl2}
As above, let $L/K$ be a cyclic extension of degree $n$, and let $v$ be a discrete valuation of $K$ such that $L_w/K_v$ $(\text{for} \ w \vert v)$ is unramified of degree $n_v$. If $\pi_w \in L^{\times}$ is such that $w(\pi_w) = 1$ and $w'(\pi_w) = 0$ for all $w' \in V^L \setminus \{ w \}$, then
$$
v(N_{L/K}(\pi_w)) = n_v \ \ \text{and} \ \ v'(N_{L/K}(\pi_w)) = 0 \ \ \text{for all} \ \ v' \in V^K \setminus \{ v \}.
$$
\end{lemma}
\begin{proof}
Let $H = \Ga(L_w/K_v) \subset G = \Ga(L/K)$. Then for $\theta \in G$ the valuation $w' := w \circ \theta$ coincides with $w$ if and only if $\theta \in H$. So,
$$
w(\theta(\pi_w)) = \left\{ \begin{array}{ccl} 1 & , & \theta \in H, \\ 0 & , & \theta \in G \setminus H. \end{array}   \right.
$$
Now, taking into account that $w \vert v$ is unramified, we obtain
$$
v(N_{L/K}(\pi_w)) = w\left( \prod_{\theta \in G} \theta(\pi_w)  \right) = \vert H \vert = n_v,
$$
proving the first assertion. The second assertion immediately follows from the fact that for any extension $w' \vert v'$ we have $w'(\pi_w) = 0$.
\end{proof}

To complete the proof of the proposition, we assume that $C = (L, \sigma, c)$ is unramified at all $v \in V$. Then $v(c) \equiv 0(\mathrm{mod}\: n_v)$ for all $v \in V$ in the above notations. Now, for each $v \in V(c)$, we pick an extension $w(v) \vert v$. Since $\Pic(L , V^L) = 0$, there exists an element $\pi_{w(v)} \in L^{\times}$ for which the corresponding principal divisor $(\pi_{w(v)})$ coincides with $w(v)$ (regarded as an element of $\Div(L , V^L)$). This means that that $w(v)(\pi_{w(v)}) = 1$ and $w'(\pi_{w(v)}) = 0$ for all $w' \in V^L \setminus \{ w(v) \}$. Consider
$$
d = \prod_{v \in V(c)} N_{L/K}(\pi_{w(v)})^{v(c)/n_v}.
$$
It follows from Lemma \ref{L:cycl2} that $v(d) = v(c)$ for all $v \in V^K$, so
$$
u := cd^{-1} \in \mathrm{U}(K , V^K).
$$
On the other hand, by construction $d \in N_{L/K}(L^{\times})$, so
$$
(L, \sigma, c) = (L, \sigma, u).
$$
So, any class in $\Br(L/K)_{V^K}$ is represented by an algebra of the form $(L, \sigma, u)$ with $u \in \mathrm{U}(K , V^K)$. Since the group
$\mathrm{U}(K , V^K)$ is finitely generated, the quotient $\mathrm{U}(K , V^K)/\mathrm{U}(K , V^K)^n$ is finite, and our assertion follows.
\end{proof}

\vskip1mm

\noindent {\bf Remark 2.7.} In the proof of Theorem \ref{T:Finite1}, we will use the conclusion of Proposition \ref{P:cycl} only in the case of a cyclic extension $L/K$ of prime degree $p$. However, the proof of the proposition in this particular case requires essentially the same argument.

\addtocounter{thm}{1}

\vskip1mm

\noindent {\bf 4. Unramified relative Brauer group for a separable extension of a prime power degree.} The goal of this
subsection is to prove the following.
\begin{prop}\label{P:primary}
Let $K$ be a field, $L$ be a separable extension of $K$ of degree $p^r$, where $p$ is a prime, and $E$ be the normal closure
of $L$ over $K$. Assume that $K$ is equipped with a set $V^K$ of discrete valuations such that for every intermediate
subfield $K \subset F \subset E$, the groups $\Pic(F , V^F)$ and $\mathrm{U}(F , V^F)$ are finitely generated. Then the
unramified relative Brauer group $\Br(L/K)_{V^K}$ is finite.
\end{prop}
\begin{proof}
For a (finite) field extension $P/F$, we let $\rho_{P/F} \colon \Br(F) \to \Br(P)$ denote the base change map. Using Azumaya algebras or properties of the residue maps (cf. \cite[Theorem 10.4]{Salt}), it is easy to see that if $x \in \Br(F)$ is unramified at a discrete valuation $v$, then $\rho_{P/F}(x)$ is unramified at any extension $w \vert v$ (we note that $x$ splits over the maximal unramified extension of $F_v$).

We will prove the proposition by induction on $r$. First, we consider the case $r = 1$. Let $\mathscr{G} = \Ga(E/K)$. Pick a Sylow $p$-subgroup $\mathscr{G}_p$ of $\mathscr{G}$, and let $F = E^{\mathscr{G}_p}$ denote the corresponding fixed subfield. Since $\mathscr{G}$ is a subgroup of the symmetric group $S_p$, the subgroup $\mathscr{G}_p$ is a cyclic group of order $p$, and therefore $E/F$ is a cyclic extension of degree $p$. We note that $E = LF = L \otimes_K F$, which
implies that $\rho_{F/K}(\Br(L/K)) \subset \Br(E/F)$. Moreover, it follows from the previous remarks that
$$
\rho_{F/K}(\Br(L/K)_{V^K}) \subset \Br(E/F)_{V^F}.
$$
Since $E/F$ is a cyclic extension and the groups $\Pic(E , V^E)$ and $\mathrm{U}(F , V^F)$ are finitely generated by assumption, the group $\Br(E/F)_{V^F}$ is finite by Proposition \ref{P:cycl}. On the other hand, since $[F : K]$ is prime to $p$, the restriction of $\rho_{F/K}$ to $\Br(L/K)$ is injective, and the finiteness of $\Br(L/K)_{V^K}$ follows.

Let now $r > 1$. Again, we pick a Sylow $p$-subgroup $\mathscr{G}_p$ of $\mathscr{G} = \Ga(E/K)$ and set $F = E^{\mathscr{G}_p}$. As above, $\rho_{F/K}$ induces an injection of $\Br(L/K)_{V^K}$ into $\Br(LF/F)_{V^F}$, so it is enough to prove the finiteness of the latter. Let $\mathscr{H}$ be the subgroup of $\mathscr{G}_p = \Ga(E/F)$ corresponding to $LF \subset E$. There exists a normal subgroup $\mathscr{N} \subset \mathscr{G}_p$ of index $p$ that contains $\mathscr{H}$, and then $M := E^{\mathscr{N}}$ is a degree $p$ cyclic extension of $F$ contained in $LF$. We have
$$
\rho_{M/F}(\Br(LF/F)_{V^F}) \subset \Br(LF/M)_{V^M}.
$$
Furthermore, the normal closure of $LF$ over $M$ is contained in $E$, and so $LF/M$ is a separable extension of degree $p^{r-1}$ that satisfies the assumptions of the proposition. By the induction hypothesis, $\Br(LF/M)_{V^M}$ is finite. On the other hand, $\mathrm{Ker}\: \rho_{M/F} \cap \Br(LF/F)_{V^F} = \Br(M/F)_{V^F}$ is finite since $M/F$ is cyclic, and the finiteness of $\Br(LF/F)_{V^F}$ follows.
\end{proof}

\vskip1mm

\noindent {\bf 5. Proof of Theorem \ref{T:Finite1}.} Let $K$ be a finitely generated field. It follows from Corollary \ref{C:primary} that it is enough to prove the finiteness of $\gen(D)$ for $D$ of a prime power degree $p^r$. Let $V^K$ be a divisorial set of places. Deleting a finite subset, we may assume that $D$ is unramified at all $v \in V^K$.
\begin{lemma}\label{L:unram}
Let $D'$ be a central division $K$-algebra that has the same maximal subfields as $D$. Then $D'$ is unramified at all $v \in V^K$.
\end{lemma}
\begin{proof}
When $p \neq \mathrm{char}\: K^{(v)}$ this follows from \cite[Lemma 2.5]{CRR1} or \cite[Lemma 2.5]{RR}, so we will now give an argument that applies to division algebras of any degree. Fix $v \in V^K$ and write
$$
D \otimes_K K_v = \mathrm{M}_{\ell}(\Delta) \ \ \text{and} \ \  D' \otimes_K K_v = \mathrm{M}_{\ell'}(\Delta').
$$
for some central division $K_v$-algebras $\Delta$ and $\Delta'$. Since $D$ and $D'$ have the same maximal subfields, it follows from \cite[Lemma 2.1]{RR} that $D \otimes_K K_v$ and $D' \otimes_K K_v$ have the same maximal \'etale $K_v$-subalgebras, and then according to \cite[Lemma 2.3]{RR} we have $\ell = \ell'$ and $\Delta$ and $\Delta'$ have the same maximal separable subfields. By our assumption, $\Delta$ is unramified at $v$, which means that the residue algebra $\Delta^{(v)}$ is a {\it central} division algebra over the residue field $K^{(v)}$, and in particular, $\Delta$ contains a maximal subfield which is (separable and) unramified over $K_v$. Then $\Delta'$ also contains such a subfield (i.e. is ``inertially split''), and hence the center $\mathscr{E}'$ of the residue algebra $(\Delta')^{(v)}$ is a separable extension of $K^{(v)}$. It is well-known (cf. \cite{Wads}) that in this case the degree $[\mathscr{E}' : K^{(v)}]$ equals the ramification index of $\Delta'$, so to prove that $\Delta'$, hence $D'$ is unramified, we need to show that $\mathscr{E}' = K^{(v)}$. Assume the contrary. Since $v$ is divisorial, the residue field $K^{(v)}$ is finitely generated. Now, using the fact that $\Delta^{(v)}$ is central over $K^{(v)}$ and applying \cite[Proposition 2.7]{RR}, we conclude that $\Delta^{(v)}$ contains a maximal separable subfield $\mathscr{P}$ which is linearly disjoint from $\mathscr{E}'$. Lifting it, we obtain a maximal subfield $P$ of $\Delta$ which is (separable and) unramified over $K_v$, with the residue field $\mathscr{P}$. Embedding this subfield into $\Delta'$, passing to the residues and comparing the degrees, we obtain that $\mathscr{P}$ must contain $\mathscr{E}'$. A contradiction, proving the lemma.
\end{proof}

To complete the proof of the theorem, fix a maximal separable subfield $L$ of $D$. Then it follows from Lemma \ref{L:unram} that $\gen(D) \subset \Br(L/K)_{V^K}$.
As we observed in subsection 1, for any finite separable extension $F/K$, the groups $\Pic(F , V^F)$ and $\mathrm{U}(F , V^F)$ are finitely generated. So, $\Br(L/K)_{V^K}$ is finite by Proposition \ref{P:primary}, and the finiteness of $\gen(D)$ follows. \hfill $\Box$

\section{Generalizations: unramified $H^3$}\label{S:G2}

The notion of the genus can be extended to algebraic groups as follows. Let $G$ be an absolutely almost simple algebraic group over a field $K$. We then define the genus $\gen_K(G)$ as the set of $K$-isomorphism classes of (inner) $K$-forms $G'$ of $G$ that have the same isomorphism classes of maximal $K$-tori as $G$. Extending the connection between the genus $\gen(D)$ of a central division $K$-algebra $D$ of degree $n$ and the unramified Brauer group ${}_n\Br(K)_V$ and its relevant subgroups, for a suitable set $V$ of discrete valuations of $K$, we have shown \cite[Theorem 5]{CRR3a}  that if $V$ satisfies (A) and the following condition

\vskip2mm

 (B) the residue field $K^{(v)}$ is finitely generated for all $v \in V$,

\vskip2mm

\noindent then given an absolutely almost simple simply connected $K$-group $G$ that has good reduction at all $v \in V$, any $K$-form of $G$ that has the same isomorphism classes of maximal $K$-tori as $G$ also has good reduction at all $v \in V$ (we refer the reader to \cite[\S1, Notations and conventions]{CRR3} for the definition of good reduction). As a consequence, one obtains that the genus $\gen_K(G)$ is finite if a (finitely generated) field $K$ can be equipped with a set $V$ of discrete valuations that satisfies (A) and (B) and is such that for any finite $S \subset V$, the set of $K$-isomorphism classes of (inner) $K$-forms of $G$ that have good reduction at all $v \in V \setminus S$ is finite. Conjecturally, this property should hold for any divisorial set of places $V$ of a finitely generated field $K$ (possibly, with some restrictions on the characteristic depending on the type of $G$). The finiteness of ${}_n\Br(K)_V$ implies the truth of this property for inner forms of type $\textsf{A}_{\ell}$ over a finitely generated field $K$ of characteristic prime to $\ell + 1$. Recently, we established this property for the spinor groups of quadratic forms, some unitary groups, and groups of type $\textsf{G}_2$ over the function field $K = k(C)$ of a smooth geometrically integral curve $C$ over a number field $k$ (see \cite{CRR3}). This relied on the finiteness of the unramified cohomology groups $H^i(K , \mu_2)_V$ for all $i \geq 1$ with coefficients in $\mu_2 = \{ \pm 1 \}$, and the case that required the most effort was $i = 3$. At this point, it does not appear that the finiteness of $H^3(K , \mu_2)_V$ is known for (substantially) more general finitely generated fields. On the other hand, the considerations in \S\ref{S:F1} suggest that the finiteness of the genus may really require the finiteness not of the entire group $H^3(K , \mu_2)_V$ but rather that of some relevant subgroup (indeed, the proof of the finiteness of $\gen(D)$ in \S\ref{S:F1} relied only on the finiteness of $\Br(L/K)_V$ and not of ${}_n\Br(K)_V$). We will postpone the details until subsection 2, and first indicate how the main assumption in \S\ref{S:F1} that $\Pic(F , V^F)$ is finitely generated for certain finite extensions $F/K$ can be upgraded to become useful in the situation at hand.

\vskip1mm

\noindent {\bf 1. Condition (T).} We begin by recalling the connection between $\Pic(K , V)$ and the ideles, which is well-known in the classical setting (cf. \cite[Ch. II, \S 17]{ANT}). So, let $K$ be a field that is equipped with a set $V$ of discrete valuations that satisfies condition (A).
For $v \in V$, we let $K_v$ denote the corresponding completion, $\mathscr{O}_v$ the valuation ring in $K_v$, and $\mathscr{U}_v = \mathscr{O}_v^{\times}$ the group of units. One defines the group of \emph{ideles} $\mathbb{I}(K , V)$ as the restricted direct product over $v \in V$ of the multiplicative groups $K^{\times}_v$  relative to the unit groups $\mathscr{U}_v$, i.e.
$$
\mathbb{I}(K , V) = \{\, (x_v) \in \prod_{v \in V} K^{\times}_v \ \vert \ x_v \in \mathscr{U}_v \ \ \text{for almost all} \ \ v \in V \,\}.
$$
Furthermore, we let
$$
\mathbb{I}^{\infty}(K , V) = \prod_{v \in V} \mathscr{U}_v
$$
denote the subgroup of \emph{integral ideles}. Since $V$ satisfies (A), one has the diagonal embedding $K^{\times} \hookrightarrow \mathbb{I}(K , V)$, the image of which will still be denoted $K^{\times}$ and called the group of \emph{principal ideles}.  Furthermore, one defines a homomorphism
$$
\nu \colon \mathbb{I}(K , V) \to \mathrm{Div}(K , V), \ \ (x_v) \mapsto \sum_{v \in V} v(x_v) \cdot v,
$$
which is easily seen to be surjective with kernel $\mathbb{I}^{\infty}(K , V)$. Since $\nu(K^{\times})$ coincides with the group of principal divisors $\mathrm{P}(K , V)$, we obtain the following.
\begin{lemma}\label{L:Id}
The map $\nu$ induces a natural identification of the quotient $\mathbb{I}(K , V) / \mathbb{I}^{\infty}(K , V) K^{\times}$ with $\mathrm{Pic}(K , V)$.
\end{lemma}

Starting with this idelic description of $\Pic(K , V)$ and imitating the definition of the ``class number" of an algebraic group defined over a number field
(cf. \cite[Ch. VIII]{PlRa}), one can associate an object similar to $\Pic(K , V)$ to any algebraic $K$-group $G$.  For this, we first define the group of {\it rational adeles} $G(\mathbf{A}(K , V))$ and its subgroups of integral and principal adeles. Fix a faithful $K$-representation $G \subset \mathrm{GL}_n$ so that one can unambiguously talk about the group of points $G(R) = G(K) \cap \mathrm{GL}_n(R)$ over a subring $R \subset K$. For $v \in V$, denote by $\mathscr{O}_{K , v}$ the valuation ring of $v$ in $K$. We then define
$$
G(\mathbf{A}(K , V)) = \{ \, (g_v) \in \prod_{v \in V} G(K) \ \vert \ g \in G(\mathscr{O}_{K , v}) \ \ \text{for almost all} \ \ v \in V \, \}.
$$
Thus, $G(\mathbf{A}(K , V))$ is the {\it restricted product} of the groups $G(K)$, one for each $v \in V$, relative to the distinguished subgroups $G(\mathscr{O}_{K , v})$ - cf. \cite[\S3.5]{PlRa} for the general definition of a restricted product. (Of course, one can also define the group of {\it full adeles} $G(\mathbb{A}(K , V))$ as the restricted products of the $G(K_v)$'s relative to the $G(\mathscr{O}_v)$'s, but rational adeles are somewhat more convenient in the current context.) It follows from (A) that any $K$-isomorphism $f \colon G \to H$ induces an isomorphism between $G(\mathscr{O}_{K , v})$ and $H(\mathscr{O}_{K , v})$ for almost all $v$, showing that the group $G(\mathbf{A}(K , V))$ is independent of the initial choice of the matrix realization $G \subset \mathrm{GL}_n$. Next, we define the group of integral (rational) adeles as
$$
G(\mathbf{A}\!^{\infty}(K , V)) = \prod_{v \in V} G(\mathscr{O}_{K , v})
$$
(note that this group does depend on the choice of the matrix realization $G \subset \mathrm{GL}_n$, and that the subgroup $G(\mathbb{A}^{\infty}(K , V))$ of integral adeles in the group of full adeles is defined similarly). Finally, due to condition (A), we have the diagonal embedding $G(K) \hookrightarrow G(\mathbf{A}(K , V))$, the image of which will still be denoted $G(K)$ and called the group of {\it principal adeles}.

\vskip2mm

\addtocounter{thm}{1}

\noindent {\bf Definition 3.2.} The set of double cosets $G(\mathbf{A}\!^{\infty}(K , V)) \backslash G(\mathbf{A}(K , V)) / G(K)$ is called the {\it class set} of $G$ (over $K$ with respect to $V$) and denoted $\mathrm{Cl}(G, K, V)$.

\vskip2mm

One can similarly consider double cosets $G(\mathbb{A}^{\infty}(K , V)) \backslash G(\mathbb{A}(K , V)) / G(K)$, which is what one does for algebraic groups over number fields (cf. \cite[Ch. VIII]{PlRa}). In this regard, we observe that the obvious equality
$$
G(\mathbb{A}^{\infty}(K , V)) \cap G(\mathbf{A}(K , V)) = G(\mathbf{A}\!^{\infty}(K , V))
$$
implies that the natural map
$$
G(\mathbf{A}\!^{\infty}(K , V)) \backslash G(\mathbf{A}(K , V)) / G(K) \longrightarrow G(\mathbb{A}^{\infty}(K , V)) \backslash G(\mathbb{A}(K , V)) / G(K)
$$
is injective, and is in fact a bijection if $G(K)$ is dense in $G(K_v)$ for all $v \in V$. The use of rational adeles instead of the full adeles enables us
to avoid additional complications  arising from the fact that $G$ may fail to have weak approximation.

\vskip2mm

\addtocounter{thm}{1}

\noindent {\bf Remark 3.3.} G.~Harder \cite[2.3]{Harder} linked double cosets with the $\check{\rm C}$ech cohomology of group schemes over Dedekind rings. We will extend this connection to the general case in \S\ref{S:DC-H1} of the Appendix.

\vskip2mm

It follows from the above discussion that for the multiplicative group $G = \mathbb{G}_m$ the elements of the class set $\mathrm{Cl}(G, K , V)$ are in bijection with the elements of the Picard group $\Pic(K , V)$. In the general case, however, $\mathrm{Cl}(G, K, V)$ may not have a natural group structure, which makes any sort of requirement of its finite generation meaningless. On the other hand, what we really used in \S\ref{S:F1} was the following consequence of the finite
generation of $\Pic(K , V)$: there exists a finite subset $S \subset V$ such that $\Pic(K , V \setminus S) = 0$. This property already easily extends to arbitrary algebraic groups.

\vskip2mm

\addtocounter{thm}{1}

\noindent {\bf Definition 3.4.} We say that an algebraic $K$-group $G$ satisfies {\it Condition} (T) with respect to a set $V$ of discrete valuation of $K$ (always assumed to satisfy condition (A)) if there exists a finite subset $S \subset V$ such that $\mathrm{Cl}(G, K, V \setminus S)$ reduces to a single element, i.e. $G(\mathbf{A}(K , V \setminus S)) = G(\mathbf{A}\!^{\infty}(K , V \setminus S)) G(K)$.

\vskip2mm

We note that while the group $G(\mathbf{A}\!^{\infty}(K , V))$, hence the class set $\mathrm{Cl}(G, K, V)$, depends on the choice of a matrix realization $G \subset \mathrm{GL}_n$, the fact that $G$ satisfies Condition (T) does not. If $K$ is a number field and $V$ is the set of all (pairwise inequivalent) nonarchimedean valuations of $K$, the class set $\mathrm{Cl}(G, K, V)$ is known to be {\it finite} for any algebraic $K$-group $G$ (cf. \cite{Borel} and \cite[Theorem 5.1]{PlRa}), which implies that Condition (T) holds in this situation. Indeed, let $g(i) = (g(i)_v)$, where $i = 1, \ldots , r$, be a finite system representatives of the double cosets $G(\mathbf{A}\!^{\infty}(K , V)) \backslash G(\mathbf{A}(K , V)) / G(K)$. One can find a finite subset $S \subset V$ such that
$$
g(i)_v \in G(\mathscr{O}_{K , v}) \ \ \text{for all} \ \ v \in V \setminus S \ \ \text{and all} \ \ i = 1, \ldots r.
$$
Let $\pi_S \colon G(\mathbf{A}(K , V)) \to G(\mathbf{A}(K , V \setminus S))$ be the natural projection. Then
$$
G(\mathbf{A}(K , V \setminus S)) = \pi_S(G(\mathbf{A}(K , V))) = \pi_S\left(\bigcup_{i = 1}^r G(\mathbf{A}\!^{\infty}(K , V)) g(i) G(K) \right)
$$
$$
= \bigcup_{i = 1}^r  G(\mathbf{A}\!^{\infty}(K , V \setminus S)) \pi_S(g(i)) G(K) = G(\mathbf{A}\!^{\infty}(K , V \setminus S)) G(K)
$$
because by our construction $\pi_S(g(i)) \in G(\mathbf{A}\!^{\infty}(K , V \setminus S))$ for all $i$, as required. No other results on condition (T) seem to be available in the literature, so in \S\ref{S:T} we will verify (T) for split reductive groups over the function field $K = k(C)$ of a smooth geometrically integral curve $C$ over a finitely generated field $k$ when $V$ the set of geometric places of $K$ --- see Theorem \ref{T:CondT}.

\vskip2mm

We will need the following consequence of the condition $\vert \mathrm{Cl}(G, K, V) \vert = 1$, which is an analogue of Lemma \ref{L:cycl2}.
\begin{lemma}\label{L:unif}
Let $V$ be a set of discrete valuations of $K$ that satisfies condition {\rm (A)}, and let $D$ be a central simple $K$-algebra. Assume that for $G = \mathrm{GL}_{1 , D}$ the class set $\mathrm{Cl}(G, K, V)$ reduces to a single element. Then, given $v \in V$ such that $D_v =D \otimes_K K_v$ is isomorphic to  $\mathrm{M}_{\ell_v}(\Delta_v)$, where $\Delta_v$ is a central division $K_v$-algebra of degree $d_v$, there exists $t_v \in D^{\times}$ satisfying
$$
v(\mathrm{Nrd}_{D/K}(t_v)) = d_v \ \ \text{and} \ \ v'(\mathrm{Nrd}_{D/K}(t_v)) = 0 \ \ \text{for all} \ \ v' \in V \setminus \{ v \}.
$$
\end{lemma}
\begin{proof}
Let $\pi_v \in K_v^{\times}$ be a uniformizer. Then for
$$
x_v = \mathrm{diag}(\pi_v, 1, \ldots , 1) \in \mathrm{M}_{\ell_v}(\Delta_v) \simeq D_v
$$
we obviously have $\mathrm{Nrd}_{D_v/K_v}(x_v) = \pi_v^{d_v}$. Using the density of $D$ in $D_v$, we  find $y_v \in D^{\times}$ such that $v(\mathrm{Nrd}_{D/K}(y_v)) = d_v$. Consider an adele $(g_w) \in G(\mathbf{A}(K , V))$ with the following components
$$
g_w = \left\{ \begin{array}{ccc} y_v & , & w = v, \\ 1 & , & w \neq v.   \end{array}  \right.
$$
Since $\vert \mathrm{Cl}(G, K, V) \vert = 1$, we can write $g = ht$ with $h \in G(\mathbf{A}\!^{\infty}(K , V))$ and $t \in G(K) = D^{\times}$. Then $t = h_v^{-1} y_v$, and since $h_v \in G(\mathscr{O}_{K , v})$, we obtain that
$$
v(\mathrm{Nrd}_{D/K}(t)) = v(\mathrm{Nrd}_{D/K}(y_v)) = d_v.
$$
On the other hand, for any $v' \neq v$ we have $t \in G(\mathscr{O}_{K , v'})$, so
$$
v'(\mathrm{Nrd}_{D/K}(t)) = 0,
$$
as required.
\end{proof}

\vskip1mm

\addtocounter{thm}{1}

\noindent {\bf Remark 3.6.} In the lemma, there was no need to specify a matrix realization of $G$ as the following is true for {\it any} realization: If $h \in G(\mathscr{O}_{K , w})$, then $w(\mathrm{Nrd}_{D/K}(h)) = 0$. This immediately follows from the observation that for such $h$, the subgroup $\langle h \rangle$ is $w$-bounded, so its image under the reduced norm, which is the restriction of a $K$-defined character of $G$, is also $w$-bounded, hence is contained in the $w$-units of $K$. We also note that the conclusion of Lemma \ref{L:unif} can be obtained under the weaker
assumption that the ``stable" class number of $G$ is one. More precisely, for $t \geqslant 1$, we consider the natural embedding $\tau_t$ of $G = \mathrm{GL}_{1 , D}$ into $G_t := \mathrm{GL}_{t , D}$ given by
\begin{equation}\label{E:Can-Emb}
x \mapsto \left( \begin{array}{cccc} x & & & \\ & 1 & & \\ & & \ddots & \\ & & & 1 \end{array} \right),
\end{equation}
and will use the same notation for the embedding of the corresponding groups of adeles. Assume that for any $g \in
G(\mathbf{A}(K, V))$, there exists $t \geqslant 1$ such that $\tau_t(g) \in G_t(\mathbf{A}\!^{\infty}(K, V)) G_t(K)$ (which of course is automatically true if $\vert \mathrm{Cl}(G_t, K, V) \vert = 1$ for some $t \geqslant 1$). Then the conclusion of Lemma \ref{L:unif} still holds.

\vskip2mm

\noindent {\bf 2. The finiteness of some subgroups of unramified $H^3$.} Let $K$ be a field of characteristic $\neq 2$ equipped with a set $V$ of discrete valuations that satisfies condition (A), and let $\mu_2 = \{ \pm 1 \}$. The goal of this subsection is to present a technique for proving the finiteness of some subgroups of the unramified cohomology group $H^3(K , \mu_2)_V$. This is relevant for proving the finiteness of the genus of simple algebraic groups of type $\textsf{G}_2$, which is expected to hold over all finitely generated fields (at least of characteristic $\neq 2, 3$), but has been established so far only in certain cases, including the situations where $K$ is a global field or the function field of a smooth geometrically integral curve over a number field (cf. \cite[\S 8]{CRR3}).

We recall the description of simple groups of type $\textsf{G}_2$. Let $G_0$ be the split group of type $\textsf{G}_2$ over a field $K$ of characteristic $\neq 2$. Then the $K$-isomorphism classes of $K$-groups of type $\textsf{G}_2$ are in a natural one-to-one correspondence with the elements of the (pointed) set $H^1(K , G_0)$ (we recall that in this case $G_0$ is naturally identified with its automorphism group $\mathrm{Aut}(G_0)$ by sending every element $g \in G_0$ to the corresponding inner automorphism $\mathrm{Int}\: g$; in what follows, we will freely use this identification). Furthermore, there is a natural map
$$
\lambda_K \colon H^1(K , G_0) \to H^3(K , \mu_2)
$$
that has the following explicit description: if $\xi \in H^1(K , G_0)$ and the twisted group $G = {}_{\xi}G_0$ is the automorphism group of the octonian algebra $\mathbb{O} = \mathbb{O}(a, b, c)$ corresponding to a triple $(a, b, c) \in (K^{\times})^3$ then
$$
\lambda_K(\xi) = \chi_a \cup \chi_b \cup \chi_c,
$$
where for $t \in K^{\times}$, we let $\chi_t$ denote the class in $H^1(K , \mu_2)$ of the cocycle given by
$$
\chi_t(\sigma) = \frac{\sigma(\sqrt{t})}{\sqrt{t}}, \ \ \sigma \in \Ga(K^{\mathrm{sep}}/K).
$$
It is well-known that $\lambda_K$ is injective (cf. \cite[Ch. III, Appendix 2, 3.3]{Serre-GC}). Furthermore, $G$ as above contains
a maximal $K$-torus of the form $T = (\mathrm{R}_{K(\sqrt{a})/K}^{(1)})^2$. So, any $G' \in \gen_K(G)$ will also contain such a torus,
and therefore is split over $K(\sqrt{a})$. It follows that $G'$ is represented by a triple of the form $(a, b', c') \in (K^{\times})^3$.

Now, let $v$ be a discrete valuation of $K$ such that $\mathrm{char}\: K^{(v)} \neq 2$. Then there is the residue map
$$
r^i_v \colon H^i(K , \mu_2) \to H^{i-1}(K^{(v)} , \mu_2) \ \ \text{for all} \ \ i \geqslant 1,
$$
and we say that $x \in H^i(K , \mu_2)$ is {\it unramified} at $v$ if its image $r^i_v(x)$ is trivial. The automorphism group $G$ of an octonian algebra $\mathbb{O}$ has good reduction at $v$ if and only if $\mathbb{O}$ can be represented by a triple $(a, b, c) \in (K^{\times})^3$ such that
$$
v(a) = v(b) = v(c) = 0.
$$
It follows that if $G = {}_{\xi} G_0$ has good reduction at $v$, then the cocycle $\lambda_K(\xi)$ is unramified at $v$. Assume now that $K$ is equipped with a set $V$ of discrete valuations that satisfies (A), (B), and such that $\mathrm{char}\: K^{(v)} \neq 2$ for all $v \in V$, and $G$ has good reduction at all $v \in V$. Then according to \cite[Theorem 5]{CRR3a}, any $G' \in \gen_K(G)$ has good reduction at all $v \in V$. Thus, if the number of cohomology classes of the form $\chi_{a} \cup \chi_{b'} \cup \chi_{c'}$ with $b', c' \in K^{\times}$ that are unramified at all $v \in V$ is finite, then $\gen_K(G)$ is also finite. We will now reformulate this in $K$-theoretic terms.

Let $K_i^M(K)$ $(i \geqslant 1)$ denote the $i$th Milnor $K$-group of the field $K$ (cf. \cite[Ch. 7]{Gille}) for the basic definitions), and set $$k_i(K) = K_i^M(K) /(2 \cdot K_i^M(K)).$$ For $a_1, \ldots , a_i \in K^{\times}$, we let $(a_1, \ldots , a_i)$ denote the corresponding symbol in $k_i(K)$, i.e. the image of $a_1 \otimes \cdots \otimes a_i$. According to Milnor's conjecture, proved by Voevodsky, the correspondence
$$
(a_1, \ldots , a_i) \mapsto \chi_{a_1} \cup \cdots \cup \chi_{a_i}
$$
extends to an isomorphism  $\kappa_i \colon k_i(K) \to H^i(K , \mu_2)$. Moreover, if $v$ is a discrete valuation of $K$ such that $\mathrm{char}\: K^{(v)} \neq 2$, then we have the following commutative diagram
$$
\xymatrix{k_i(K) \ar[r]^{\kappa_i} \ar[d]_{\partial^i_v} & H^i(K , \mu_2) \ar[d]^{r_v^i} \\ k_{i-1}(K^{(v)}) \ar[r]^{\kappa_{i-1}} & H^{i-1}(K^{(v)} , \mu_2)}
$$
where ${\partial}^i_v$ is the residue map in Milnor $K$-theory (cf. \cite[7.5]{Gille}). In particular, a symbol
$(a_1, \ldots , a_i) \in k_i(K)$ is unramified at $v$ (i.e., has the trivial image under ${\partial}^i_v$) if and only if the cohomology class $\chi_{a_1} \cup \cdots \cup \chi_{a_i} \in H^i(K , \mu_2)$ is unramified as defined above.  Putting all this together, we obtain the following.
\begin{lemma}\label{L:G2finite}
Let $K$ be a field that is equipped with a set $V$ of discrete valuations that satisfies {\rm (A)} and {\rm (B)}, and is such that $\mathrm{char}\: K^{(v)} \neq 2$
for all $v \in V$. Let $G$ be the automorphism group of an octonion algebra $\mathbb{O} = \mathbb{O}(a, b, c)$, and assume that $G$ has good reduction at all
$v \in V$. If the number of symbols $(a, b', c') \in k_3(K)$, where $b' , c' \in K^{\times}$, that are unramified at all $v \in V$ is finite, then $\gen_K(G)$ is also finite.
\end{lemma}

Now, fix $a \in K^{\times}$ and consider the map
$$
f_a \colon {}_2\Br(K) = H^2(K , \mu_2) \to H^3(K , \mu_2), \ \ \alpha \mapsto \chi_a \cup \alpha,
$$
noting that in terms of the above identifications given by $\kappa_i$ $(i = 2, 3)$, this map is equivalent to
$$
k_2(K) \to k_3(K), \ \ (b , c) \mapsto (a , b , c).
$$
It follows from Lemma \ref{L:G2finite} that to prove the finiteness of $\gen_K(G)$, it would be enough to prove the finiteness of
$\mathrm{Im}\: f_a \cap H^3(K , \mu_2)_V$. Since the finiteness of ${}_2\Br(K)_V$, where $V$ is a divisorial set of places of a finitely generated field $K$, is already known \cite{CRR2}, one can attempt to prove the finiteness of $\mathrm{Im}\: f_a \cap H^3(K , \mu_2)_V$ by showing that this intersection is commensurable with $f_a({}_2\Br(K)_V)$ (we note that if $v(a) = 0$ and $\mathrm{char}\: K^{(v)} \neq 2$, then $\chi_a$ is unramified, hence $f_a({}_2\Br(K)_V) \subset H^3(K , \mu_2)_V$). So far, we have not been able to prove this in the general case, however, as the following result shows, the finiteness of some subgroups of $\mathrm{Im}\: f_a \cap H^3(K , \mu_2)_V$ can be established assuming Condition (T) for algebraic groups associated with some quaternion algebras. We will formulate it in terms of symbols rather than cohomology classes given by the corresponding cup-products as this simplifies the notations.
\begin{thm}\label{T:101}
Fix $a , b_1, \ldots b_r \in K^{\times}$ and consider the subgroup $\Delta \subset k_3(K)$ consisting of elements of the form $$t(c_1, \ldots , c_r) = \sum_{i = 1}^r  (a, b_i, c_i) \ \ \text{for} \ \ c_1, \ldots , c_r \in K^{\times}.$$ Given a subset $V' \subset V$, we let $\Delta_{\mathrm{ur} , V'}$ denote
the subgroup of $\Delta$ consisting of elements that are unramified at all $v \in V'$, and for a nonempty subset $J$ of  $\{ 1, \ldots , r \}$, let $D_J$ be the quaternion algebra
$$
\left( \frac{a , b_J}{K} \right) \ \ \text{where} \ \ b_J = \prod_{i \in J} b_i.
$$
Assume that Condition {\rm (T)} holds for the groups $G_J = \mathrm{GL}_{1 , D_J}$ for all $J \subset \{ 1, \ldots , r \}$. Then there exists a subset $V' \subset V$ with finite complement such that $a , b_1, \ldots , b_r \in \mathrm{U}(K , V')$ and any $x \in \Delta_{\mathrm{ur} , V'}$ is of the form $x = t(c_1, \ldots , c_r)$ for some $c_1, \ldots , c_r \in \mathrm{U}(K , V')$. Consequently, if $\mathrm{U}(K , V)$, or, equivalently, $\mathrm{U}(K , V')$, is finitely generated, then $\Delta_{\mathrm{ur} , V'}$ (hence also $\Delta_{\mathrm{ur} , V}$) is finite.
\end{thm}
\begin{proof}
We can pick a subset $V' \subset V$ with finite complement such that

\vskip2mm

\noindent $\bullet$ $a, b_1, \ldots , b_r \in \mathrm{U}(K , V')$;

\vskip1mm

\noindent $\bullet$ $\vert \mathrm{Cl}(G_J, K, V') \vert = 1$ for all nonempty subsets $J \subset \{ 1, \ldots , r \}$.

\vskip2mm

\noindent We will repeatedly use the following property: {\it if $a, b, c \in K^{\times}$  are such that $c$ is a reduced norm from the quaternion algebra
$\displaystyle \left( \frac{a , b}{K} \right)$, then $(a , b , c) = 0$ in $k_3(K)$} (cf. \cite[12.1]{MS}). Now, fix $i \in \{ 1, \ldots , r \}$. Then it follows from Lemma
\ref{L:unif} that for any $v \in V'$, there exists $t_v \in K^{\times}$ which is a reduced norm from the quaternion algebra $\displaystyle \left( \frac{a , b_i}{K} \right)$ such that $v(t_v) = 2$ and $v'(t_v) = 0$ for all $v' \in V' \setminus \{ v \}$. Multiplying $c_i$ by a suitable power of $t_v$, we can construct an element
$c'_i \in K^{\times}$ such that
$$
(a, b_i, c_i) = (a, b_i, c'_i) , \ \ v(c'_i) = 0 \ \text{or} \ 1, \ \ \text{and} \ \ v'(c'_i) = v'(c_i) \ \ \text{for all} \ \ v' \in V' \setminus \{ v \}.
$$
Iterating this procedure, we may assume that $v(c_i) = 0$ or $1$ for all $v \in V'$ and all $i$.

Set
$$
V(c_1, \ldots , c_r) = \{ v \in V' \ \mid \ v(c_i) \neq 0 \ \ \text{for some} \ \ i = 1, \ldots r \}.
$$
We will induct on $d := \vert V(c_1, \ldots , c_r) \vert$. If $d = 0$, then $c_1, \ldots , c_r \in \mathrm{U}(V')$, and there is nothing to prove.
Suppose $d > 0$, and let $v \in V(c_1, \ldots , c_r)$. Then $J := \{ i \ \mid \ v(c_i) \neq 0 \}$ is nonempty. Using an explicit description of the residue
map $\partial_v$ in Milnor $K$-theory (cf. \cite[7.1]{Gille}) and
taking into account that the elements $a, b_1, \ldots , b_r$ are $v$-units, we find that
$$
{\partial}_v(t(c_1, \ldots , c_r)) = \sum_{i \in J} (\overline{a} , \overline{b_i}) = (\overline{a} , \overline{b_J}) \in k_2(K^{(v)}),
$$
where the bar denotes the image in the residue field $K^{(v)}$. Since by assumption $t(c_1, \ldots , c_r)$ is unramified and $\mathrm{char}\: K^{(v)} \neq 2$, we conclude that the quaternion algebra $\displaystyle \left( \frac{\overline{a} , \overline{b_J}}{K^{(v)}} \right)$ is trivial. Then by Hensel's lemma, the quaternion algebra $\displaystyle \left( \frac{a , b_J}{K_v}  \right) = D_J \otimes_K K_v$ is also trivial. So, it follows from Lemma \ref{L:unif} that there exists
$\pi_v \in K^{\times}$ which is a reduced norm from $D_J$ and satisfies $v(\pi_v) = 1$ and $v'(\pi_v) = 0$ for all $v' \in V' \setminus \{ v \}$. Then
\begin{equation}\label{E:X200}
(a, b_J, \pi_v) = 0 = \sum_{i \in J} (a, b_i, \pi_v) \ \ \text{in} \ \  k_3(K).
\end{equation}
Set $c'_i = c_i \pi_v^{-1}$ for $i \in J$, and $c'_i = c_i$ for $i \in \{ 1, \ldots , r \} \setminus J$. Then it follows from (\ref{E:X200}) that
$$
t(c_1, \ldots , c_r) = t(c_1, \ldots , c_r) - (a, b_J, \pi_v) = t(c'_1, \ldots , c'_r).
$$
Clearly,
$$
V(c'_1, \ldots , c'_r) = V(c_1, \ldots , c_r) \setminus \{ v \},
$$
so $t(c'_1, \ldots , c'_r) = t(u_1, \ldots , u_r)$ for some $u_1, \ldots , u_r \in \mathrm{U}(K , V')$ by the induction hypothesis, and the required fact follows.
\end{proof}

\vskip2mm

\addtocounter{thm}{1}

\noindent {\bf Remark 3.9.} It follows from Remark 3.6 that the assertion of Theorem \ref{T:101} remains valid if one assumes that each of the groups $G_J$ satisfies the ``stable'' version of Condition (T). More precisely, for $t \geqslant 1$, we let $\tau_{J , t} \colon G_J \to G_{J , t} := \mathrm{GL}_{t , D_J}$
denote the canonical embedding given by (\ref{E:Can-Emb}). Then the ``stable'' version of (T) requires that there be a subset $V' \subset V$ with finite complement such that for any $g \in G_J(\mathbf{A}(K , V'))$, there exists $t \geqslant 1$ (depending on $g$) such that $\tau_{J , t}(g) \in G_{J , t}(\mathbf{A}\!^{\infty}(K , V')) G_{J , t}(K)$. This obviously holds if $\mathrm{GL}_{\ell , D_J}$ satisfies Condition (T) for {\it some} $\ell \geqslant 1$.

\vskip2mm

\begin{cor}\label{C:Unram7}
Let $K = k(C)$ be the function field of a smooth geometrically integral curve over a field $k$ of characteristic $\neq 2$, and let $V_0$ be the set of places of $K$ associated with closed points of $C$. As in Theorem \ref{T:101}, fix $a, b_1, \ldots , b_r \in K^{\times}$ and consider the subgroup $\Delta$ of $k_3(K)$ consisting of elements of the form $t(c_1, \ldots , c_r)$ for all $c_1, \ldots , c_r \in K^{\times}$. Assume that Condition {\rm (T)} holds for $G_J$ for all subsets $J \subset \{1, \ldots , r \}$ in the above notations. Then for the subgroup of $V_0$-unramified elements  $\Delta_{\mathrm{ur} , V_0}$ we have
$$
[\Delta_{\mathrm{ur} , V_0} : \Delta_{\mathrm{ur} , V_0} \cap \Delta_0] < \infty,
$$
where $\Delta_0$ is formed by elements of the form $t(c_1, \ldots , c_r)$ with $c_1, \ldots , c_r \in k^{\times}$.
\end{cor}
\begin{proof}
According to the theorem there exists a subset $V' \subset V_0$ with finite complement such that $\Delta_{\mathrm{ur} , V_0}$ is contained in the subgroup $\Delta_1$ formed by $t(c_1, \ldots , c_r)$  with $c_1, \ldots , c_r \in \mathrm{U}(K , V')$. But the quotient $\mathrm{U}(K , V')/k^{\times}$ is finitely generated, which implies that the quotient $\Delta_1/\Delta_0$ is finite, and our claim follows.
\end{proof}

\vskip2mm

\noindent {\bf 3. An application.}
Let $C$ be a smooth geometrically integral projective curve over a field $k$ such that $C(k) \neq \emptyset$, let $K = k(C)$ be its function field, and let $V_0$ be the set of discrete valuations of $K$ associated with closed points of $C$. Let $J$ be the Jacobian of $C$, and assume that the
2-torsion of $J$ is $k$-rational, i.e. ${}_2J \subset J(k)$. Then the natural map $\Br(k) \to \Br(K)$ is injective and there exists a homomorphism
$$
\nu \colon {}_2J \otimes_{\Z} H^1(k , \mu_2) \longrightarrow {}_2\Br(K)
$$
so that
$$
{}_2\Br(K)_{V_0} = {}_2\Br(k) \oplus \mathscr{B}(C) \ \ \text{where} \ \ \mathscr{B}(C) = \mathrm{Im}\: \nu,
$$
cf. \cite[\S 6]{CRR2}. The homomorphism $\nu$ has the following explicit description. Let $u_1, \ldots , u_{2g}$ (where $g$ is the genus of $C$) be a basis of ${}_2J$ over $\Z/2\Z$. Identifying of $J(k)$ with $\Pic^0(C)$, we can find degree zero divisors $D_1, \ldots , D_{2g} \in \mathrm{Div}^0(C)$ whose images in $\Pic^0(C)$ are $u_1, \ldots , u_{2g}$. Furthermore, there are functions $h_1, \ldots , h_{2g} \in K^{\times}$ such that $2 D_i$ is the divisor of $h_i$ for $i = 1, \ldots , 2g$. By Kummer theory, any element of $H^1(k , \mu_2)$ has the form $\chi_c$ for some $c \in k^{\times}$
uniquely determined modulo ${k^{\times}}^2$, in our previous notations; here, however, we will write $\chi_{k , c}$ instead of $\chi_c$ to indicate the base field explicitly. Then $\nu$ is described by
$$
\nu(u_i \otimes \chi_{k , c}) = \chi_{K , h_i} \cup \chi_{K , c},
$$
cf. \cite[Proposition 6.1]{CRR2}. Thus, the ``nonconstant'' part $\mathscr{B}(C)$ of ${}_2\Br(K)_{V_0}$ in the case at hand is the following:
$$
\mathscr{B}(C) = \left\{ \sum_{i = 1}^{2g} \chi_{K , h_i} \cup \chi_{K , c_i} \ \vert \ c_1, \ldots , c_{2g} \in k^{\times}     \right\};
$$
recall that $\chi_{K , h} \cup \chi_{K , c}$ corresponds to the quaternion algebra $\displaystyle \left( \frac{h , c}{K} \right)$. (Explicit computations of ${}_2\Br(C)$ for elliptic curves $C$ can be found in \cite{CG}.)

\vskip2mm

Now, if we fix $a \in K^{\times}$ and consider the map $f_a \colon {}_2\Br(K) \to H^3(K , \mu_2)$, $\alpha \mapsto \chi_{K , a} \cup \alpha$, then in terms of the identification $H^3(K , \mu_2) \simeq k_3(K)$ given by $\kappa_3^{-1}$, the image $f_a(\mathscr{B}(C))$ is represented by elements of
the form
$$
\sum_{i = 1}^{2g} (a, h_i, c_i) \ \ \text{for} \ \ c_1, \ldots , c_{2g} \in k^{\times}.
$$
Theorem \ref{T:101} applies to the set $\Delta$ of such elements, and we obtain the following.
\begin{prop}\label{P:102}
Let $C$ be a smooth projective geometrically integral curve over a finitely generated field $k$ of characteristic $\neq 2$ that has a rational
point and whose Jacobian has rational 2-torsion. Let $K = k(C)$ be its function field, and assume that Condition {\rm (T)} holds for any group of the form
$G = \mathrm{GL}_{1 , D}$, where $D$ is a quaternion algebra over $K$, with respect to a divisorial set $V$ of discrete valuations of $K$. Then in the above notations, for any $a \in K^{\times}$ the intersection
$$
f_a(\mathscr{B}(C)) \cap H^3(K , \mu_2)_V
$$
is finite.
\end{prop}

\vskip2mm

\noindent {\bf Remark 3.12.} (i) A remark similar to Remark 3.9 can be stated also with regard to the use of Condition (T) in Corollary \ref{C:Unram7}
and Proposition \ref{P:102}.

(ii) The description of ${}_2\Br(K)_{V_0}$ that we gave above can be generalized. Namely, let $C$ be a smooth geometrically integral projective curve over a field $k$ such that $C(k) \neq \emptyset$, let $K = k(C)$ be its function field, and let $V_0$ be the set of discrete valuations of $K$ associated with closed points of $C$. Fix an integer $n \geqslant 1$ prime to $\mathrm{char}\: k$ and assume that the group $\mu_n$ of $n$th roots of unity is contained in $k$. Furthermore, let $J$ be the Jacobian of $C$, and assume that the $n$-torsion of $J$ is rational, i.e. ${}_nJ \subset J(k)$. Then for any $\ell \geqslant 2$ the natural map $H^{\ell}(k , \mu_n) \to H^{\ell}(K , \mu_n)$ is injective and there exists a homomorphism
$$
\nu_{n , \ell} \colon {}_nJ \otimes_{\Z} H^{\ell-1}(k , \mu_n) \longrightarrow H^{\ell}(K , \mu_n)
$$
so that
$$
H^{\ell}(K , \mu_n)_{V_0} = H^{\ell}(k , \mu_n) \oplus \mathrm{Im}\: \nu_{n , \ell}.
$$
The construction of $\nu_{n , \ell}$ repeats almost verbatim the construction of $\nu$ above (cf. \cite[\S 6]{CRR2} for the details).

\section{Condition (T) for some reductive groups}\label{S:T}

\begin{thm}\label{T:CondT}
Let $C$ be a smooth geometrically integral curve over a finitely generated field $k$ with function field $K = k(C)$, and let $V$ be the set of discrete valuations of $K$ associated with closed points of $C$. Then Condition {\rm (T)} with respect to $V$ holds for any connected reductive split $K$-group $G$.
\end{thm}
\begin{proof}
Since the class set $\mathrm{Cl}(\mathbb{G}_m, K, V)$ can be identified with
$\Pic(K , V)$, which is finitely generated because $K$ is finitely generated, we can find a subset $V' \subset V$ with finite complement such that $\mathrm{Cl}(\mathbb{G}_m, K, V')$ reduces to a single element. Now, fix a maximal split $K$-torus $T$ of $G$. Since $T \simeq \mathbb{G}_m^d$, we can reduce $V'$ further (by a finite set) if necessary to insure that $\vert \mathrm{Cl}(T, K, V') \vert = 1$. We will show that then $\vert \mathrm{Cl}(G, K, V') \vert = 1$. For this we will rely on some considerations that use a form of strong approximation.

For any algebraic $K$-group $H$, the group of rational adeles $H(\mathbf{A}(K , V))$ can be topologized as a restricted product (cf. \cite[3.5]{PlRa}); note
that this topology induces the product topology on $H(\mathbf{A}\!^{\infty}(K , V))$. We then say that $H$ has {\it strong approximation} with respect to $V$ if $H(K)$ is dense in $H(\mathbf{A}(K , V))$. First, let $H = \mathbb{G}_a$. Without loss of generality, we can assume that $C$ is affine and then
consider the coordinate ring $k[C]$. For a closed point $P \in C$, the valuation ideal $\mathfrak{m}_P$ in $k[C]$ of the corresponding valuation $v_P$ is maximal, hence for $P_1 \neq P_2$, the ideals $\mathfrak{m}_{P_1}$ and $\mathfrak{m}_{P_2}$ are relatively prime. Applying the Chinese Remainder Theorem, we obtain that the diagonal embedding
$$
k[C] \hookrightarrow \prod_{v \in V} k[C]
$$
is dense when the product is given the product topology (we note that the topology on the factor corresponding to $v = v_P$  is defined by the powers of the ideal $\mathfrak{m}_P$). Then it is easy to show that $K$ is dense in $\mathbf{A}(K , V)$, i.e. $\mathbb{G}_a$ has strong approximation (with respect to $V$ as above if $C$ is affine).

For a root $\alpha$ in the root system $\Phi(G , T)$, we let $U_{\alpha}$ denote the corresponding unipotent root subgroup, and let $G(K)^+$ denote the subgroup of $G(K)$ generated by $U_{\alpha}(K)$ for all $\alpha \in \Phi(G , T)$ (recall that since $G$ is split, we actually have $G(K)^+ = G(K)$ when $G$ is in addition semi-simple and simply connected). Set
$$
\mathcal{G} = G(\mathbf{A}(K , V)) \bigcap \prod_{v \in V} G(K)^+.
$$
Clearly, the union of the products
$$
\mathcal{G}(S) = \prod_{v \in S} G(K)^+  \ \ (\subset \mathcal{G}),
$$
taken over all finite subset $S \subset V$, is dense in $\mathcal{G}$. Since for each $\alpha \in \Phi(G , T)$, the group $U_{\alpha} \simeq \mathbb{G}_a$ has strong approximation with respect to $V$, we conclude that the diagonal embedding $G(K)^+ \hookrightarrow \mathcal{G}$ is dense. Now, it is well-known (and follows, for example, from the Bruhat decomposition) that $G(K) = T(K)G(K)^+$.
Clearly, an arbitrary double coset $$G(\mathbf{A}\!^{\infty}(K , V))\, x \, G(K), \ \ x \in G(\mathbf{A}(K , V))$$ contains an element that lies in one of the finite products $$
G(S) = \prod_{v \in S} G(K) \ \ (\subset G(\mathbf{A}(K , V))),
$$
which implies that it has a representative of the form $tg$ with $t \in
T(\mathbf{A}(K , V))$ and $g \in \mathcal{G}$. But the subgroup $$t^{-1} (\mathcal{G} \cap G(\mathbf{A}\!^{\infty}(K , V))) t = \mathcal{G} \cap (t^{-1} G(\mathbf{A}\!^{\infty}(K , V))t)$$
is open in $\mathcal{G}$, so the density of $G(K)^+$ in $\mathcal{G}$ implies that
$$
\mathcal{G} = t^{-1} (\mathcal{G} \cap G(\mathbf{A}\!^{\infty}(K , V))) t G(K)^+.
$$
Writing $g = g_1g_2$ with $g_1 \in t^{-1} (\mathcal{G} \cap G(\mathbf{A}\!^{\infty}(K , V))) t$ and $g_2 \in G(K)^+$, we obtain that
$$
G(\mathbf{A}\!^{\infty}(K , V)) x G(K) = G(\mathbf{A}\!^{\infty}(K , V)) (tg_1t^{-1})tg_2 G(K) = G(\mathbf{A}\!^{\infty}(K , V)) t G(K).
$$
Thus, $$G(\mathbf{A}(K , V)) = G(\mathbf{A}\!^{\infty}(K , V)) T(\mathbf{A}(K , V)) G(K).$$ Projecting to $G(\mathbf{A}(K , V'))$ and taking into account that $T(\mathbf{A}(K , V')) = T(\mathbf{A}\!^{\infty}(K , V')) T(K)$, we obtain that $\vert \mathrm{Cl}(G, K, V') \vert = 1$, as claimed.
\end{proof}

\vskip2mm

It would interesting to consider some other situations in which Condition (T) holds  for a given reductive group $G$ over a field $K$ which is
equipped with a set $V$ of discrete valuations that satisfies condition (A). As we have seen, it would be useful to have Condition (T) for groups over finitely generated fields with respect to divisorial sets of places of those fields; already the case of groups of the form $G = \mathrm{GL}_{\ell , D}$, where $D$ is a quaternion algebra, would be very interesting. Another important case is where $K = k(C)$ is the function field of a smooth geometrically integral curve over a (finitely generated) field $k$ and $V$ is the set of discrete valuations associated with closed points of $C$.
We would like to point out that Condition (T) does hold for all algebraic tori over finitely generated fields with respect to divisorial sets of places.
\begin{prop}\label{P:CondT}
Let $K$ be a finitely generated field, and let $V^K$ be a divisorial set of places of $K$. Then any $K$-torus $T$ satisfies Condition {\rm (T)}.
\end{prop}
\begin{proof}
Let $L$ be the splitting field of $T$. As we have already seen, $T$ satisfies Condition (T) over $L$ with respect to $V^L$ (the set of all extensions of places from $V^K$ to $L$). So, reducing $V^K$ by a finite set, we may assume that
\begin{equation}\label{E:Fact}
T(\mathbf{A}(L , V^L)) = T(\mathbf{A}\!^{\infty}(L , V^L)) T(L).
\end{equation}
Using this, we will now construct an injective homomorphism
$$
\lambda \colon T(\mathbf{A}(K , V^K))/T(\mathbf{A}\!^{\infty}(K , V^K)) T(K) \longrightarrow H^1(L/K , U) \ \ \text{where} \ \ U = T(\mathbf{A}\!^{\infty}(L , V^L)) \cap T(L).
$$
Take $t \in T(\mathbf{A}(K , V^K))$, and using (\ref{E:Fact}) write it as $t = t_1t_2$ with $t_1 \in T(\mathbf{A}\!^{\infty}(L , V^L))$ and $t_2 \in T(L)$. Considering the natural action of $\mathscr{G} = \Ga(L/K)$ on $T(\mathbf{A}(L , V^L))$ that leaves $T(\mathbf{A}\!^{\infty}(L , V^L))$ invariant, for every $\sigma \in \mathscr{G}$, we obtain that
\begin{equation}\label{E:ksi}
\xi(\sigma) := \sigma(t_2)t_2^{-1} = \sigma(t_1)^{-1} t_1 \in T(\mathbf{A}\!^{\infty}(L , V^L)) \cap T(L) = U.
\end{equation}
Clearly, $\xi(\sigma)$, $\sigma \in \mathscr{G}$,  defines a cocycle $\xi$ on $\mathscr{G}$ with values in $U$. The class of $\xi$ in $H^1(\mathscr{G} , U) = H^1(L/K , U)$ is easily seen to depend only on $t$ but not on the choice of a particular factorization $t = t_1t_2$, and we will denote it by $\xi_t$. Moreover, the correspondence $t \mapsto \xi_t$ defines a homomorphism
$$
\tilde{\lambda} \colon T(\mathbf{A}(K , V^K)) \longrightarrow H^1(L/K , U).
$$
Now, suppose $t \in \mathrm{Ker}\: \tilde{\lambda}$. Choose a factorization $t = t_1t_2$ as above. Then there exists $u \in U$ such that $\xi(\sigma) = \sigma(u)u^{-1}$ for all $\sigma \in \mathscr{G}$ where $\xi(\sigma)$ is given by (\ref{E:ksi}). Rearranging, we see that
$$
t_1u \in T(\mathbf{A}\!^{\infty}(L , V^L))^{\mathscr{G}} = T(\mathbf{A}\!^{\infty}(K , V^K)) \ \ \text{and} \ \ u^{-1} t_2 \in T(L)^{\mathscr{G}} = T(K),
$$
and consequently $t = (t_1u)(u^{-1}t_2) \in T(\mathbf{A}\!^{\infty}(K , V^K))T(K)$, which proves the inclusion $\mathrm{Ker}\: \tilde{\lambda} \subset T(\mathbf{A}\!^{\infty}(K , V^K)) T(K)$. The opposite inclusion is obvious, so $\tilde{\lambda}$ descends to a required injective homomorphism $\lambda$.

The finite generation of $\mathrm{U}(L , V^L)$ yields the finite generation of $U$, which implies that the group $H^1(L/K , U)$ is finite. Since $\lambda$ is injective, we conclude that the quotient $$T(\mathbf{A}(K , V^K))/T(\mathbf{A}\!^{\infty}(K , V^K)) T(K)$$ is finite, and then Condition (T) follows immediately.
\end{proof}

Using Proposition \ref{P:CondT} and adapting the proof of Theorem \ref{T:CondT}, one can extend this theorem to {\it quasi-split} groups. No more general groups have been considered so far, so we would like indicate how the strategy of reducing the general case to the split case that we employed in the proof of Proposition \ref{P:CondT} can potentially be extended to noncommutative groups. So, let $G$ be an algebraic group defined over a field $K$ equipped with a set $V$ of discrete valuations that satisfies (A). Suppose we are given a finite Galois extension $L/K$ with Galois group $\mathscr{G}$. We then let $V'$ denote the set of all extensions of valuations from $V$ to $L$ (note that $V'$ also satisfies (A)), and let $R' = \bigcap_{w \in V'} \mathscr{O}_{L , w}$. Furthermore, assume that
$$
G(\mathbf{A}(L , V')) = G(\mathbf{A}\!^{\infty}(L , V')) G(L).
$$
Then there exists an injective map
$$
\lambda \colon G(\mathbf{A}\!^{\infty}(K , V)) \backslash G(\mathbf{A}(K , V)) / G(K) \longrightarrow H^1(\mathscr{G} , G(R')).
$$
Note that $\mathrm{Im}\: \lambda$ is contained in the kernel of the global-to-local map
$$
\iota \colon H^1(\mathscr{G} , G(R')) \longrightarrow \prod_{w} H^1(\mathscr{G}(w) , G(\mathscr{O}_{L , w})),
$$
where for each $v \in V$ we pick one extension $w \in V'$ and let $\mathscr{G}(w)$ denote the decomposition group of $w$. Over number fields, all this is due to Rohlfs \cite{Rohl} (cf. \cite[8.4]{PlRa}). Of course, the implementation of this approach for verifying Condition (T) would require the finiteness of $H^1(\mathscr{G} , G(R'))$ or at least of $\mathrm{Ker}\: \iota$. We will present some relevant computations in \S\ref{S:Descent} of the Appendix.

\vskip2mm

It is well-known that the presence of some form of strong approximation is helpful for analyzing double cosets of adele groups -- see the treatment of groups over number fields in \cite[Ch. VIII]{PlRa} and the proof of Theorem \ref{T:CondT} above. While strong approximation is well-understood over global fields (see \cite{R-SA} for a recent survey), not much seems to be known over more general fields. So, we would like to formulate the following

\vskip2mm

\noindent {\bf Question 4.3.} {\it Let $K = k(C)$ be the function field of a smooth geometrically integral curve $C$ over a field $k$,
let $V$ be the set of discrete valuations of $K$ associated with closed points of $C$, and let $G$ be a semi-simple simply connected $K$-group. In what situations does there exist a subset $V' \subset V$ with finite complement such that the closure of $G(K)$ in $G(\mathbf{A}(K , V'))$ is either all of $G(\mathbf{A}(K , V'))$ (i.e. $G$ has strong approximation with respect to $V'$) or at least a ``big'' subgroup thereof? }

\vskip2mm

If $k$ is algebraically closed, then $G$ is necessarily quasi-split, and an easy modification of the argument used in the proof of Theorem \ref{T:CondT} shows that $G$ does have strong approximation with respect to $V' = V$ if $C$ is affine (note that in this case $G(K)^+ = G(K)$). However, the cases where $k$ is a local field or an infinite finitely generated field and the group $G$ is $K$-anisotropic appear to be almost completely open --- to the best of our knowledge, the only available result is due to Yamasaki \cite{Yamasaki}, which gives strong approximation for the group $\mathrm{SL}_{1,D}$, where $D$ is a quaternion division algebra over the field $\R(x)$ of rational functions over $\R$.

\vskip8mm

\centerline{\sc Appendix}

\vskip5mm

\section{Double cosets of the adele group and $\check{\rm C}$ech $\check{H}^1$}\label{S:DC-H1}

The goal of this section is to link double cosets of adele groups with a certain $\check{\rm C}$ech 1-cohomology set by constructing an injective map from the latter to the former (see Proposition \ref{P:DoubleC}), which extends the result of G.~Harder \cite[2.3]{Harder} for group schemes over Dedekind rings. Our construction is explicit and will be used later in the the proof of Proposition \ref{P:Ap3}. Besides, it does not use any results from sheaf theory making our exposition self-contained.

We begin with a brief review of the construction of $\check{\rm C}$ech 1-cohomology set for a sheaf of groups. So, let $\mathcal{G}$ be a sheaf of (noncommutative) groups on a topological space $X$. We begin by briefly recalling the construction of the (pointed) set $\check{H}^1(X , \mathcal{G})$ of $\check{\rm C}$ech 1-cohomology. Let $\mathfrak{U} = \{ U_i \}_{i \in I}$ be an open cover of $X$. We let $\check{Z}^1(\mathfrak{U} , \mathcal{G})$ denote the corresponding set of $\check{\rm C}$ech 1-cocycles, i.e. the families $\{ g_{ij} \}_{i,j \in I}$ with $g_{ij} \in \mathcal{G}(U_i \cap U_j)$ satisfying the cocycle relation
\begin{equation}\label{E:1}
(g_{jk} \vert U_{ijk}) (g_{ik} \vert U_{ijk})^{-1} (g_{ij} \vert U_{ijk}) = 1, \ \ \text{i.e.} \ \ g_{ik} = g_{ij} g_{jk} \ \ \text{on} \ \ U_{ijk} := U_i \cap U_j \cap U_k.
\end{equation}
Two cocycles $\{ g_{ij} \}$ and $\{ g'_{ij} \}$ in $\check{Z}^1(\mathfrak{U} , \mathcal{G})$ are called {\it equivalent} if there exist $s_i \in \mathcal{G}(U_i)$ such that
$$
g'_{ij} = (s_i \vert (U_i \cap U_j)) g_{ij} (s_j \vert (U_i \cap U_j))^{-1} \ \ \text{for all} \ \ i,j \in I.
$$
One easily checks that this is indeed an equivalence relation on $\check{Z}^1(\mathfrak{U} , \mathcal{G})$, and the corresponding set of equivalence classes will be  denoted $\check{H}^1(\mathfrak{U} , \mathcal{G})$; it is a {\it pointed} set whose distinguished element is the class of the trivial cocycle $\{ g_{ij} \}$ defined by $g_{ij} = 1$ for all $i , j \in I$. Furthermore, let $\mathfrak{V} = \{ V_j \}_{j \in J}$ be another open cover of $X$. We write $\mathfrak{U} \leqslant \mathfrak{V}$ if there exists a {\it refinement map} $\tau \colon J \to I$ such that $V_j \subset U_{\tau(j)}$ for all $j \in J$. One then defines a map $\check{Z}^1(\mathfrak{U} , \mathcal{G}) \to \check{Z}^1(\mathfrak{V} , \mathcal{G})$ by sending $g = \{ g_{i_1i_2} \}_{i_1,i_2 \in I} \in \check{Z}^1(\mathfrak{U} , \mathcal{G})$ to $g' = \{ g'_{j_1j_2} \}_{j_1,j_2 \in J} \in \check{Z}^1(\mathfrak{V} , \mathcal{G})$ defined by
$$
g'_{j_1 j_2} = g_{\tau(j_1) \tau(j_2)} \vert (V_{j_1} \cap V_{j_2}).
$$
One easily check that this map is compatible with the equivalence relations on $\check{Z}^1(\mathfrak{U} , \mathcal{G})$ and $\check{Z}^1(\mathfrak{V} , \mathcal{G})$ yielding a map $\check{H}^1(\mathfrak{U} , \mathcal{G}) \to \check{H}^1(\mathfrak{V} , \mathcal{G})$. It turns out that this map is in fact independent of the choice of the refinement map $\tau$. Indeed, let $\sigma \colon J \to I$ be another refinement map. Given a cocycle $g = \{ g_{i_1 i_2} \}_{i_1 , i_2 \in I} \in \check{Z}^1(\mathfrak{U} , \mathcal{G})$,  we  set
$$
s_j = g_{\sigma(j) \tau(j)} \vert V_j \ \ \text{for} \ \ j \in J
$$
(note that $g_{\sigma(j) \tau(j)} \in \mathcal{G}(U_{\sigma(j)} \cap U_{\tau(j)})$ and $U_{\sigma(j)} \cap U_{\tau(j)} \supset V_j$).
Then it follows from the cocycle relation (\ref{E:1}) that
$$
g_{\sigma(j_1) \sigma(j_2)} \vert (V_{j_1} \cap V_{j_2}) = (s_{j_1} \vert (V_{j_1} \cap V_{j_2})) \cdot (g_{\tau(j_1) \tau(j_2)} \vert (V_{j_1} \cap V_{j_2})) \cdot (s_{j_2} \vert (V_{j_1} \cap V_{j_2}))^{-1}.
$$
This means that the cocycles
$$
g' = \{ g'_{j_1 j_2} = g_{\tau(j_1) \tau(j_2)} \vert (V_{j_1} \cap V_{j_2})\} \ \ \text{and} \ \ g'' = \{ g''_{j_1 j_2}
= g_{\sigma(j_1) \sigma(j_2)} \vert (V_{j_1} \cap V_{j_2})\} \ \in \ \check{Z}^1(\mathfrak{V} , \mathcal{G}),
$$
obtained using the refinement maps $\tau$ and $\sigma$, respectively, are equivalent, hence define the same element in $\check{H}^1(\mathfrak{V} , \mathcal{G})$. Thus, for any two open covers $\mathfrak{U}$ and $\mathfrak{V}$ of $X$ such that $\mathfrak{U} \leqslant \mathfrak{V}$, there is a canonical map of pointed sets
\begin{equation}\label{E:ref}
\tau^{\mathfrak{V}}_{\mathfrak{U}} \colon \check{H}^1(\mathfrak{U} , \mathcal{G}) \longrightarrow \check{H}^1(\mathfrak{V} , \mathcal{G}).
\end{equation}
Moreover, the sets $\check{H}^1(\mathfrak{U} , \mathcal{G})$ together with the maps $\tau^{\mathfrak{V}}_{\mathfrak{U}}$ for $\mathfrak{U} \leqslant \mathfrak{V}$ form a direct system over the partially ordered directed set of all covers, and one defines the $\check{\rm C}$ech cohomology set of $X$ with coefficients in $\mathcal{G}$ by
$$
\check{H}^1(X , \mathcal{G}) = \lim_{\longrightarrow} \check{H}^1(\mathfrak{U} , \mathcal{G}).
$$

\vskip2mm

We will now specialize to the following situation. Let $R$ be a noetherian integral domain that is integrally closed in its field of fractions\footnotemark \ $K$,
and let $G$ be an affine flat group scheme over $R$. We can consider a sheaf of groups $\mathcal{G}$ on $X = \mathrm{Spec}\: R$ given by
$$
\mathcal{G}(U) := G(\mathcal{O}_X(U)) \ \ \text{for any (Zariski) open} \ \ U \subset X,
$$
where $\mathcal{O}_X$ is the structure sheaf of $X$. To emphasize the role of $G$, we will denote the corresponding $\check{\rm C}$ech cohomology set by $\check{H}^1(X , G)$.

\footnotetext{In particular, $R$ is a Krull domain, cf. \cite[Ch. VII, \S1, n$^{\circ}$ 3]{Bour-CA}}

We let $\mathrm{P}$ denote the set of height one primes of $R$. For each $\mathfrak{p} \in \mathrm{P}$, we let $v_{\mathfrak{p}}$ denote the corresponding discrete valuation of $K$ and set $V = \{ v_{\mathfrak{p}} \, \vert \, \mathfrak{p} \in \mathrm{P} \}$. Conversely, for $v \in V$ we let $\mathfrak{p}_v \in
\mathrm{P}$ denote the corresponding prime, and let $\mathscr{O}_{K , v}$ be the valuation ring of $v$ in $K$ (we note that $\mathscr{O}_{K , v_{\mathfrak{p}}}$ coincides with the localization $R_{\mathfrak{p}}$). We note that $V$ satisfies condition (A) - see \cite[Ch. VII, \S1, Theorem 4]{Bour-CA}.
It is also well-known that $R = \bigcap_{v \in V} \mathscr{O}_{K , v}$ (cf. \cite[Ch. II, Proposition 6.3A]{Hart}). More generally, for any nonzero $a \in R$ and the corresponding localization $R_a$, we have
\begin{equation}\label{E:102}
R_a = \bigcap_{v \in V \setminus V(a)} \mathscr{O}_{K,v} \ \ \text{where} \ \ V(a) := \{ v \in V \: \vert \: v(a) \neq 0 \}.
\end{equation}
Throughout the remainder of this section, $K$, $R$ and $V$ will remain fixed, so in order to simply our notations, we will write $\mathscr{O}_v$ instead of $\mathscr{O}_{K , v}$. Furthermore, we will denote the  group of rational adeles $G(\mathbf{A}(K,V))$ (cf. \S\ref{S:G2}) simply by $G(\mathbf{A})$; its subgroup of integral adeles will be denoted $G(\mathbf{A}\!^{\infty})$,  i.e. $G(\mathbf{A}\!^{\infty}) = \prod_{v \in V} G(\mathscr{O}_v)$. Our goal is to prove the following.


\begin{prop}\label{P:DoubleC}
There exists a natural injective map $$f \colon \check{H}^1(X , G) \longrightarrow G(\mathbf{A}\!^{\infty}) \backslash G(\mathbf{A}) / G(K).$$
\end{prop}

This proposition suggests the following analogue of Condition (T) for $\check{H}^1$, which is likely to be true in all situations.

\addtocounter{thm}{1}

\vskip2mm

\noindent {\bf Conjecture 5.2.} {\it Let $R$ be a finitely generated $\Z$-algebra which is an integral domain integrally closed in its field of fractions,
and $X = \mathrm{Spec}\: R$. Given an affine reductive group scheme $G$ of finite type over $R$, there exists an open subset  $U \subset X$ such that the restriction map $$\check{H}^1(X , G) \to \check{H}^1(U , G)$$ is trivial (or even $\check{H}^1(U , G) = 1$).}

\vskip2mm

We begin by first fixing an open cover $\mathfrak{U} = \{ U_i \}_{i \in I}$ and constructing a natural map
$$
f_{\mathfrak{U}} \colon \check{H}^1(\mathfrak{U} , G) \longrightarrow G(\mathbf{A}\!^{\infty}) \backslash G(\mathbf{A}) / G(K).
$$
As long as the cover $\mathfrak{U}$ remains fixed, we will, for simplicity, suppress $\mathfrak{U}$ in our notations (such as $f_{\mathfrak{U}}$), and will reinstate the subscript $\mathfrak{U}$ once we begin working with different covers. Let $g = \{ g_{ij} \} \in \check{Z}^1(\mathfrak{U} , G)$ be a $\check{\rm C}$ech 1-cocycle. Fix $i_0 \in I$ and $h_{i_0} \in G(K)$, and for all $i \in I$ set
$$
h_i = g_{ii_0} h_{i_0}.
$$
It follows from (\ref{E:1}) that  $g_{i_0i_0} = 1$, hence this does not alter $h_0$, and moreover we have $g_{ij} = h_i h_j^{-1}$ for all $i , j \in I$, meaning that the family $\{ h_i \}$ provides a trivialization of $g$ as a cocycle with values in the constant sheaf associated with $G(K)$. 

Since $X$ is quasi-compact, we can find a finite subset $I_0 \subset I$ so that the sets $U_i$ for $i \in I_0$ already cover $X$. Pick a finite subset $V_0 \subset V$ such that for any $v \in V \setminus V_0$, the elements $h_i$ $(i \in I_0)$ all belong to $G(\mathscr{O}_v)$. Take an arbitrary $v \in V \setminus V_0$, and let $i \in I$ be such that $\mathfrak{p}_v \in U_i$. There exists $j \in I_0$ for which $\mathfrak{p}_v \in U_j$. We have $g_{ij} \in G(\mathcal{O}_X(U_i \cap U_j)) \subset G(\mathscr{O}_v)$, so $h_i = g_{ij} h_j \in G(\mathscr{O}_v)$. Thus, $h_i \in G(\mathscr{O}_v)$ for any $v \in V \setminus V_0$ such that $\mathfrak{p}_v \in U_i$. 

Now, let $\phi \colon V \to I$ be any map with the property $\mathfrak{p}_v \in U_{\phi(v)}$ for all $v \in V$. It follows from the above discussion that 
there is an adele $\tilde{f}(g; i_0, h_{i_0}, \phi) \in G(\mathbf{A})$  with the  components
$$
\tilde{f}(g; i_0, h_{i_0}, \phi)_v = h_{\phi(v)} \ \ \text{for} \ \ v \in V.
$$
We then set 
$$
f(g; i_0, h_{i_0}, \phi) = G(\mathbf{A}\!^{\infty}) \tilde{f}(g; i_0, h_{i_0}, \phi) G(K) \in G(\mathbf{A}\!^{\infty}) \backslash G(\mathbf{A}) / G(K).
$$
Our goal is to show that the double coset $f(g; i_0, h_{i_0}, \phi)$ depends only on the cocycle $g$.

\vskip2mm

\noindent 1$^{\circ}$. $f(g; i_0, h_{i_0}, \phi)$ {\it is independent of} $\phi$. Indeed, let $v \in V$ and let $i , i' \in I$ be such that $\mathfrak{p}_v \in U_i$ and $U_{i'}$. It follows from (\ref{E:1}) that $g_{ii_0} = g_{ii'} g_{i'i_0}$, so multiplying this relation by $h_{i_0}$, we see that $$h_{i} = g_{ii'}h_{i'} \ \  \text{with} \ \  g_{ii'} \in G(\mathcal{O}_X(U_i \cap U_{i'})) \subset G(\mathscr{O}_v).$$ So, for any two maps $\phi, \phi'$ as above, the elements  $\tilde{f}(g; i_0, h_{i_0}, \phi)$ and $\tilde{f}(g; i_0, h_{i_0}, \phi')$ lie in the same right coset modulo $G(\mathbf{A}\!^{\infty})$, hence our claim.

\vskip2mm

So, we will fix some map $\phi$ as above for the rest of the construction of $f_{\mathfrak{U}}$.

\vskip2mm

\noindent 2$^{\circ}$. $f(g; i_0, h_{i_0}, \phi)$ {\it is independent of} $h_{i_0}$. Indeed, take some other $h'_{i_0} \in G(K)$, and set $t = h_{i_0}^{-1} h'_{i_0}$. Then clearly  $h'_i = g_{ii_0} h'_{i_0}$ satisfies $h'_i = h_i t$, implying that
$$
\tilde{f}(g; i_0, h'_{i_0}, \phi) = \tilde{f}(g; i_0, h_{i_0}, \phi) t.
$$
So, $\tilde{f}(g; i_0, h'_{i_0}, \phi)$ and $\tilde{f}(g; i_0, h_{i_0}, \phi)$ lie in the same left coset modulo $G(K)$, and our claim follows.

\vskip2mm

\noindent 3$^{\circ}$. $f(g; i_0, h_{i_0}, \phi)$ {\it is independent of} $i_0$. Let $i'_0 \in I$ be another index. It follows from the relation $g_{ii'_0} = g_{ii_0} g_{i_0 i'_0}$ that
$$
\tilde{f}(g; i_0, h_{i_0}, \phi) = \tilde{f}(g; i'_0, g_{i_0 i'_0}^{-1} h_{i_0}, \phi),
$$
and consequently, $f(g; i_0, h_0, \phi) = f(g; i'_0, h'_{i'_0}, \phi)$ for any choice of $h'_{i'_0} \in G(K)$ in view of item 2$^{\circ}$.

\vskip2mm

So, we will write $f(g)$ instead of $f(g; i_0, h_0, \phi)$. Next, we show that if $g' \in \check{Z}^1(\mathfrak{U} , G)$ is equivalent to $g$ then $f(g') = f(g)$. The fact that the cocycles are equivalent means that there exist $s_i \in G(\mathcal{O}_X(U_i))$ for $i \in I$ such that
\begin{equation}\label{E:5}
g'_{ij} = s_i g_{ij} s_j^{-1} \ \ \text{for all} \ \ i , j \in I.
\end{equation}
Set $h'_{i_0} = s_{i_0} h_{i_0}$. Then
$$
h'_i := g'_{i i_0} h'_{i_0} = s_i g_{i i_0} s_{i_0}^{-1} (s_{i_0} h_{i_0}) = s_i h_i \ \ \text{for all} \ \ i \in I.
$$
Then for every $v \in V$ we have
$$
h'_{\phi(v)} = s_{\phi(v)} h_{\phi(v)} \in G(\mathcal{O}_X(U_{\phi(v)})) h_{\phi(v)} \subset G(\mathcal{O}_v) h_{\phi(v)}.
$$
It follows that $\tilde{f}(g; i_0, h_{i_0}, \phi)_v = h_{\phi(v)}$ and $\tilde{f}(g'; i_0, h'_{i_0}, \phi)_v = h'_{\phi(v)}$ lie
in the same right coset modulo $G(\mathcal{O}_v)$. Thus, $\tilde{f}(g; i_0, h_{i_0}, \phi)$ and $\tilde{f}(g; i_0, h'_{i_0}, \phi)$
lie in the same right coset modulo $G(\mathbf{A}\!^{\infty})$, hence $f(g') = f(g)$.

\vskip2mm

The above discussion implies that $f$ descends to a well-defined map
$$
f_{\mathfrak{U}} \colon \check{H}^1(\mathfrak{U} , G) \longrightarrow G(\mathbf{A}\!^{\infty}) \backslash G(\mathbf{A}) /G(K).
$$
\begin{lemma}\label{L:princ}
If all the $U_i$'s in $\mathfrak{U}$ are principal open sets, then $f_{\mathfrak{U}}$ is injective.
\end{lemma}
\begin{proof}
Suppose $U_i$ is the principal open subset of $X$ corresponding to $a_i \in R$, i.e. $U_i = \mathrm{Spec}\: R_{a_i}$. Applying (\ref{E:102}), we see that
\begin{equation}\label{E:10}
\mathcal{O}_X(U_i) = \bigcap_{v \in V_i} \mathscr{O}_{v},
\end{equation}
where $V_i$ consists of those $v \in V$ for which $\mathfrak{p}_v \in U_i$.

Suppose now $g , g' \in \check{Z}^1(\mathfrak{U} , G)$ are such that $f(g') = f(g)$. Fix $\phi$, $i_0$ and $h_{i_0}$ as above. Then there exist $s \in G(\mathbf{A}\!^{\infty})$ and $t \in G(K)$ such that
\begin{equation}\label{E:6}
\tilde{f}(g'; i_0, h_{i_0}, \phi) = s \tilde{f}(g; i_0, h_{i_0}, \phi) t = s \tilde{f}(g; i_0, h_{i_0}t, \phi)
\end{equation}
(the last equality follows from the computations in item 2$^{\circ}$ above). Setting
$$
h_i = g_{i i_0} h_{i_0}, \ \ h'_{i_0} = h_{i_0} \ \ \text{and} \ \ h'_i = g'_{ii_0} h'_{i_0} \ \ \text{for} \ \ i \in I,
$$
we will have
$$
g_{ij} = (h_it)(h_jt)^{-1} \ \ \text{and} \ \ g'_{ij} = h'_i (h'_j)^{-1}.
$$
So, to show that $g$ and $g'$ define the same element of $\check{H}^1(\mathfrak{U} , G)$, i.e. there exist $s_i \in G(\mathcal{O}_X(U_i))$ satisfying (\ref{E:5}), it suffices to show that
\begin{equation}\label{E:7}
h'_i (h_i t)^{-1} \in G(\mathcal{O}_X(U_i)) \ \ \text{for all} \ \ i \in I.
\end{equation}
Fix an $i \in I$, and let $\phi_i \colon V \to I$ be a map such that $\mathfrak{p}_v \in U_{\phi(v)}$ for all $v \in V$ and $\phi_i(V_i) = \{ i \}$. The argument in item 2$^{\circ}$ above shows that $\tilde{f}(g; i_0, h_{i_0}t, \phi)$ and $\tilde{f}(g; i_0, h_{i_0}t, \phi_i)$ (resp., $\tilde{f}(g'; i_0, h_{i_0}, \phi)$ and $\tilde{f}(g; i_0, h_{i_0}, \phi_i)$) lie in the same right coset modulo $G(\mathbf{A}\!^{\infty})$, so (\ref{E:6}) implies that
\begin{equation}\label{E:8}
\tilde{f}(g'; i_0, h_{i_0}, \phi_i) = s_i \tilde{f}(g; i_0, h_{i_0}t, \phi_i) \ \ \text{with} \ \ s_i \in G(\mathbf{A}\!^{\infty}).
\end{equation}
Since for $v \in V_i$ we have
$$
\tilde{f}(g'; i_0, h_{i_0}, \phi_i)_v = h'_i \ \ \text{and} \ \ \tilde{f}(g; i_0, h_{i_0}t, \phi_i) = h_i t,
$$
(\ref{E:7}) follows from (\ref{E:8}) and (\ref{E:10}).
\end{proof}

\vskip2mm

Let now $\mathfrak{V} = \{ V_j \}_{j \in J}$ be a refinement of the cover $\mathfrak{U} = \{ U_i \}_{i \in I}$. Then there exists a canonical map (\ref{E:ref}), and we claim that
\begin{equation}\label{E:200}
f_{\mathfrak{U}} = f_{\mathfrak{V}} \circ \tau^{\mathfrak{V}}_{\mathfrak{U}}.
\end{equation}
For this we need to fix a refinement map $\tau \colon J \to I$ such that $V_j \subset U_{\tau(j)}$ for all $j \in J$, and recall that $\tau^{\mathfrak{V}}_{\mathfrak{U}}$ is defined by sending the class of  $g = (g_{i_1 i_2}) \in \check{Z}^1(\mathfrak{U} , G)$ to the class of $g' = (g'_{j_1 j_2}) \in \check{Z}^1(\mathfrak{V} , G)$ defined by
$$
g'_{j_1 j_2} := g_{\tau(j_1) \tau(j_2)}.
$$
Fix $j_0 \in J$ and set $i_0 = \tau(j_0)$. Furthermore, fix $h'_{j_0} = h_{i_0} \in G(K)$, pick an arbitrary function $\phi' \colon V \to J$ as above, and set $\phi = \tau \circ \phi'$. To prove that $f_{\mathfrak{U}}(g) = f_{\mathfrak{V}}(g')$, we consider $h_i = g_{i i_0} h_{i_0}$ for $i \in I$ and
$$
h'_j = g'_{j j_0} h'_{j_0} = g_{\tau(j) \tau(j_0)} h_{\tau(j_0)} = h_{\tau(j)} \ \ \text{for} \ \ j \in J.
$$
Since  $\phi(v) = \tau(\phi'(v))$ for all $v \in V$, we obtain that  $h'_{\phi'(v)} = h_{\phi(v)}$, and consequently
$$
\tilde{f}_{\mathfrak{V}}(g'; j_0, h'_{j_0}, \phi')_v = h'_{\phi'(v)} = h_{\phi(v)} = \tilde{f}_{\mathfrak{U}}(g; i_0, h_{i_0}, \phi)_v.
$$
So,
$$
\tilde{f}_{\mathfrak{V}}(g'; j_0, h'_{j_0}, \phi') = \tilde{f}_{\mathfrak{U}}(g; i_0, h_{i_0}, \phi),
$$
and therefore $f_{\mathfrak{V}}(g') = f_{\mathfrak{U}}(g)$. This proves (\ref{E:200}), which implies that the maps $f_{\mathfrak{U}}$ are compatible,   hence give rise to a required map
$$
f \colon \check{H}^1(X , G) \longrightarrow G(\mathbf{A}\!^{\infty}) \backslash G(\mathbf{A}) / G(K).
$$
Since for every open cover $\mathfrak{U}$ there exists a cover $\mathfrak{V}$ consisting of principle open sets and such that $\mathfrak{U} \leqslant \mathfrak{V}$, we conclude from Lemma \ref{L:princ} that $f$ is injective, completing the proof of Proposition \ref{P:DoubleC}.

\section{The case $G = \mathrm{GL}_n$}\label{S:GLn}

We will continue using the notations introduced in the previous section. In particular, $R$ will denote a noetherian integral domain integrally closed in its field of fractions $K$. We set $X = \mathrm{Spec}\: R$, and let $\mathrm{P}$ be the set of height one primes of $R$ and $V$ the associated set of discrete valuations of $K$.
The group of rational adeles $G(\mathbf{A}(K , V))$ will be denoted simply by $G(\mathbf{A})$. In fact, in this section we will focus exclusively on the case $G = \mathrm{GL}_n$. We begin by reviewing the known facts that in this case the sets $\check{H}^1(X , G)$ and $G(K) \backslash G(\mathbf{A}) / G(\mathbf{A}\!^{\infty})$ can be interpreted in terms of finitely generated {\it projective} and {\it reflexive} $R$-modules.

For $n \geqslant 1$, let $\texttt{Proj}_n(R)$ denote the set of isomorphism classes of finitely generated projective $R$-modules of rank $n$.

\begin{prop}\label{P:Ap1}
{\rm (cf. \cite[\S 11]{Milne})} For $G = \mathrm{GL}_n$, there exists a natural bijection {\rm $$\alpha \colon \check{H}^1(X , G) \to  \text{\texttt{Proj}}_n(R).$$}
\end{prop}

We recall the construction of $\alpha^{-1}$. Let $M$ be a finitely generated projective $R$-module of rank $n$. Then there exists an open cover $\mathfrak{U} = \{ U_i \}_{i \in I}$ of $X$ by principal open sets such that for each $i \in I$, we have an isomorphism of $\mathcal{O}_X(U_i)$-modules
\begin{equation}\label{E:Isom7}
\varphi_i \colon M \otimes_{R} \mathcal{O}_X(U_i) \longrightarrow \mathcal{O}_X(U_i)^n
\end{equation}
(here $\mathcal{O}_X$ denotes the structure sheaf on $X$). For $i , j \in I$, we let
\begin{equation}\label{E:cocycle}
\psi_{ij} := (\varphi_i \otimes_{\mathcal{O}_X(U_j)} \mathrm{id}_{\mathcal{O}_X(U_i \cap U_j)}) (\varphi^{-1}_j \otimes_{\mathcal{O}_X(U_i)} \mathrm{id}_{\mathcal{O}_X(U_i \cap U_j)}) \ \in \ G(\mathcal{O}_X(U_i \cap U_j)).
\end{equation}
 One shows that the family $\psi(M) := (\psi_{ij})$ belongs to $\check{Z}^1(\mathfrak{U} , G)$, and the corresponding cohomology class $[\psi(M)] \in \check{H}^1(X , G)$ is well-defined. Then $\alpha^{-1}$ sends the isomorphism class of $M$ to $[\psi(M)]$.

\vskip2mm

Next, for an $R$-module $M$, we let $M^* = \mathrm{Hom}_{R\text{-mod}}(M , R)$. There is then a natural homomorphism of $R$-modules $c_M \colon M \to M^{**} = (M^*)^*$, and $M$ is called {\it reflexive} if this homomorphism is an isomorphism. Clearly, finitely generated projective modules are reflexive. It is easy to see that a reflexive module is necessarily torsion-free, hence can be regarded as an $R$-submodule of the $K$-vector space $W = M \otimes_{R} K$. We will assume henceforth that $M$ is finitely generated, and then $W$ can be identified with $K^n$, where $n$ is the {\it rank} of $M$,   making $M$ an $R$-{\it lattice} (i.e., a finitely generated $R$-submodule containing a $K$-basis) in $W = K^n$. One then checks that $M^*$ can be identified with the {\it dual lattice} of $M$, which is defined to be the set of all $x^* \in W^*$ such that $x^*(x) \in R$ for all $x \in M$ (and is indeed a lattice in $W^*$). Using the canonical isomorphism $c_W \colon W \to W^{**}$ we will routinely identify $M^{**}$ with a lattice in $W$.


For $\mathfrak{p} \in \mathrm{P}$, let $R_{\mathfrak{p}}$ be the corresponding localization of $R$, and for any $R$-submodule $M \subset W$ set $\mathrm{M}_{\mathfrak{p}} := R_{\mathfrak{p}} M \subset W$, noting that if $M$ is an $R$-lattice, then $\mathrm{M}_{\mathfrak{p}}$ is an $R_{\mathfrak{p}}$-lattice in $W$. One shows (cf. \cite[Ch. VII, \S4, Theorem 2]{Bour-CA}) that $M$ is reflexive if and only if
\begin{equation}\label{E:Refl1}
M = \bigcap_{\mathfrak{p} \in \mathrm{P}} M_{\mathfrak{p}}.
\end{equation}
Indeed, since $R = \bigcap_{\mathfrak{p} \in \mathrm{P}} R_{\mathfrak{p}}$, for any $R$-lattice $M$ in $W$ we have $$M^* = \bigcap_{\mathfrak{p} \in \mathrm{P}} M^*_{\mathfrak{p}},$$ where $M^*_{\mathfrak{p}}$ consists of those $x^* \in W^*$ that satisfy $x^*(x) \in R_{\mathfrak{p}}$ for all $x \in M_{\mathfrak{p}}$ (equivalently, for all $x \in M$). Applying this to $M^*$, we obtain $$M^{**} = \bigcap_{\mathfrak{p} \in
\mathrm{P}} M^{**}_{\mathfrak{p}}.$$ But for every $\mathfrak{p} \in \mathrm{P}$, the localization $R_{\mathfrak{p}}$ is a DVR, hence a PID. Thus, $M_{\mathfrak{p}}$ is a free $R_{\mathfrak{p}}$-module, hence $M^{**}_{\mathfrak{p}} = M_{\mathfrak{p}}$. We see that for {\it any} $R$-lattice $M$ in $W$, one has
\begin{equation}\label{E:Refl2}
M^{**} = \bigcap_{\mathfrak{p} \in \mathrm{P}} M_{\mathfrak{p}},
\end{equation}
and therefore (\ref{E:Refl1}) is  equivalent to $M = M^{**}$. In particular, a reflexive lattice is uniquely determined by its localizations. Furthermore, it is easy to see that for any lattice $M \subset W$, the lattice $M^* \subset W^*$ is always reflexive. In particular, the lattice $M^{**} \subset W$ is reflexive. Comparing this with (\ref{E:Refl2}), we conclude that for any lattice $M \subset W$, the intersection $\bigcap_{\mathfrak{p} \in \mathrm{P}} M_{\mathfrak{p}}$ is a reflexive lattice in $W$.

\vskip2mm

Let $\texttt{Refl}_n(R)$ denote the set of isomorphism classes of reflexive $R$-lattices in $W = K^n$.

\begin{prop}\label{P:Ap2}
{\rm (cf. Bourbaki \cite[Ch. VII, \S4, n$^{\circ}$ 3, Remark]{Bour-CA})} For $G = \mathrm{GL}_n$, there is a natural bijection {\rm $$\beta \colon G(K) \backslash G(\mathbf{A}) /  G(\mathbf{A}\!^{\infty}) \to \texttt{Refl}_n(R)$$}.
\end{prop}

The proof uses the following result.

\begin{thm}\label{T:Refl1}
{\rm (Bourbaki \cite[Ch. VII, \S4, Thm. 3]{Bour-CA})} Let $M$ be an $R$-lattice in $W = K^n$. Then

\vskip2mm

\noindent {\rm (1)} \parbox[t]{15.5cm}{for any other lattice $N \subset W$ we have $N_{\mathfrak{p}} = M_{\mathfrak{p}}$ for almost all $\mathfrak{p} \in \mathrm{P}$;}

\vskip2mm

\noindent {\rm (2)} \parbox[t]{15.5cm}{given $R_{\mathfrak{p}}$-lattices $N(\mathfrak{p}) \subset W$ for $\mathfrak{p} \in \mathrm{P}$ such that $N(\mathfrak{p}) = M_{\mathfrak{p}}$ for almost all $\mathfrak{p}$, the intersection $$N = \bigcap_{\mathfrak{p} \in \mathrm{P}} N(\mathfrak{p})$$ is a reflexive $R$-lattice such that $N_{\mathfrak{p}} = N(\mathfrak{p})$ for all $\mathfrak{p} \in \mathrm{P}$.}
\end{thm}

\vskip2mm

{\it Sketch of proof.} (1) Since $M$ and $N$ are finitely generated as $R$-modules and span $W$ over $K$, there exist nonzero $x , y \in
R$ such that $xN \subset M$ and $yM \subset N$. However both $x$ and $y$ are units in $R_{\mathfrak{p}}$ for almost all $\mathfrak{p} \in \mathrm{P}$ (see \cite[Ch. VII, \S 1, Theorem 4]{Bour-CA}), and the required fact follows.

\vskip1mm

(2) We can assume that $N(\mathfrak{p}) \subset M_{\mathfrak{p}}$ for all $\mathfrak{p} \in \mathrm{P}$. Then $N \subset M^{**}$, hence finitely generated since $R$ is noetherian. It is easy to see that for every $x \in W$ there exists $a \in R$ such that $a x \in N$, implying that $N$ is a lattice in $W$. We only need to show that $N_{\mathfrak{p}} = N(\mathfrak{p})$ for all $\mathfrak{p} \in \mathrm{P}$.

First, we observe that for any two distinct $\mathfrak{p}_1 , \mathfrak{p}_2 \in \mathrm{P}$ and the multiplicative set $S = (R \setminus \mathfrak{p}_1) (R \setminus \mathfrak{p}_2)$, one has
\begin{equation}\label{E:Refl3}
S^{-1} R = K.
\end{equation}
Indeed, otherwise, the ring $S^{-1}R$ would contain a nonzero prime ideal that would correspond to a nonzero prime ideal $\mathfrak{p}$ of $R$ having empty intersection with $S$. But then $\mathfrak{p} \subset \mathfrak{p}_1 \cap \mathfrak{p}_2$, which is impossible because $\mathfrak{p}_1$ and $\mathfrak{p}_2$ are distinct height one prime ideals.

Now, let $\mathfrak{p}_1, \ldots , \mathfrak{p}_h \in \mathrm{P}$ be those prime ideals for which $N(\mathfrak{p}_i) \neq M_{\mathfrak{p}_i}$. We will now show that for the lattice
$$
Q := M \cap N(\mathfrak{p}_1) \cap \cdots \cap N(\mathfrak{p}_h),
$$
one has $Q_{\mathfrak{p}} = N(\mathfrak{p})$ for all $\mathfrak{p}$. Indeed, if $\mathfrak{p} \in \mathrm{P} \setminus \{\mathfrak{p}_1, \ldots , \mathfrak{p}_h \}$ then
$$
Q_{\mathfrak{p}} = M_{\mathfrak{p}} \cap N(\mathfrak{p}_1)_{\mathfrak{p}} \cap \cdots \cap N(\mathfrak{p}_h)_{\mathfrak{p}} = M_{\mathfrak{p}} = N(\mathfrak{p})
$$
since according to (\ref{E:Refl3}) for any $i = 1, \ldots , h$ we have $(R_{\mathfrak{p}_i})_{\mathfrak{p}} = K$ hence $N(\mathfrak{p}_i)_{\mathfrak{p}} = W$. Let now $\mathfrak{p} = \mathfrak{p}_i$. Then $N(\mathfrak{p}_j)_{\mathfrak{p}_i} = W$ for $j \neq i$ and $N(\mathfrak{p}_i)_{\mathfrak{p}_i} = N(\mathfrak{p}_i)$, so
$$
Q_{\mathfrak{p}_i} = M_{\mathfrak{p}_i} \cap N(\mathfrak{p}_i) = N(\mathfrak{p}_i).
$$
Then $N$ coincides with $Q^{**} = \bigcap_{\mathfrak{p} \in \mathrm{P}} Q_{\mathfrak{p}}$ which is reflexive, and $N_{\mathfrak{p}} = Q_{\mathfrak{p}} = N(\mathfrak{p})$ for all $\mathfrak{p} \in \mathrm{P}$, as required.

\vskip2mm

{\it Proof of Proposition \ref{P:Ap2}.} For $G = \mathrm{GL}_n$, the adele group $G(\mathbf{A})$ naturally acts on the set $\mathscr{L} = \{ L \}$ of all reflexive $R$-lattices in $W = K^n$ as follows. Let $L_0 = R^n$ be the standard lattice. Given $g = (g_{\mathfrak{p}}) \in G(\mathbf{A})$ (where we write $g_{\mathfrak{p}}$ instead of $g_{v_{\mathfrak{p}}}$) and a lattice $L \in \mathscr{L}$, for almost all $\mathfrak{p} \in \mathrm{P}$ we have
$$
g_{\mathfrak{p}} \in G(R_{\mathfrak{p}}) \ \ \ \text{and} \ \ \ L_{\mathfrak{p}} = L_{0 \mathfrak{p}}.
$$
It follows that $g_{\mathfrak{p}}(L_{\mathfrak{p}}) = L_{\mathfrak{p}}$ for almost all $\mathfrak{p} \in \mathrm{P}$, so according to Theorem \ref{T:Refl1},
$$
g(L) := \bigcap_{\mathfrak{p} \in \mathrm{P}} g_{\mathfrak{p}}(L_{\mathfrak{p}})
$$
is a reflexive lattice in $W$ (i.e., $g(L) \in \mathscr{L}$), and in fact is the {\it only} reflexive lattice whose localization at $\mathfrak{p}$ is $g_{\mathfrak{p}}(L_{\mathfrak{p}})$ for all $\mathfrak{p} \in \mathrm{P}$. By looking at localizations, it is easy to see that the map $(g , L) \mapsto g(L)$ defines an action of $G(\mathbf{A})$ on $\mathscr{L}$. Furthermore, since the localization $L_{\mathfrak{p}}$ of any lattice $L \in \mathscr{L}$ is a free $R_{\mathfrak{p}}$-module for every prime $\mathfrak{p} \in \mathrm{P}$, this action is transitive. Since the the stabilizer of $L_0$ is $G(\mathbf{A}\!^{\infty})$, we obtain a natural bijection
$$
G(\mathbf{A})/G(\mathbf{A}\!^{\infty}) \longrightarrow \mathscr{L}, \ \ \ gG(\mathbf{A}\!^{\infty}) \mapsto g(L_0).
$$
Finally, any $R$-module isomorphism $L \to L'$ between two reflexive $R$-lattice is induced by an element of $G(K)$, and the proposition follows. \hfill $\Box$.

\vskip2mm

We will now amalgamate the bijections $\alpha$ and $\beta$ described in Propositions \ref{P:Ap1} and \ref{P:Ap2} into a~commutative diagram.
\begin{prop}\label{P:Ap3}
The diagram {\rm
\begin{equation}\label{E:Diagr}
\xymatrix{\check{H}^1(X , G) \ar[rr]^{\theta} \ar[d]_{\alpha} & & G(K) \backslash G(\mathbf{A}) / G(\mathbf{A}\!^{\infty})  \ar[d]^{\beta} \\
\texttt{Proj}_n(R)  \ar[rr]^{\iota} & & \texttt{Refl}_n(R)}
\end{equation}}

\noindent where $\theta = \varepsilon \circ f$, with $f \colon \check{H}^1(X , G) \to G(\mathbf{A}\!^{\infty}) \backslash G(\mathbf{A}) / G(K)$ the injection
constructed in Proposition~\ref{P:Ap1}, and $\varepsilon \colon G(\mathbf{A}\!^{\infty}) \backslash G(\mathbf{A}) / G(K) \longrightarrow G(K) \backslash G(\mathbf{A}) / G(\mathbf{A}\!^{\infty})$ the map given by $$G(\mathbf{A}\!^{\infty}) x G(K) \mapsto G(K) x^{-1} G(\mathbf{A}\!^{\infty}),$$
and $\iota$  is the natural embedding, commutes.
\end{prop}
\begin{proof}
Let $M$ be a projective $R$-module of rank $n$. As we indicated after the statement of Proposition~\ref{P:Ap1}, the map $\alpha^{-1}$ takes the isomorphism class of $M$ to the cohomology class $[\psi(M)]$ of the cocycle $\psi(M) = (\psi_{ij})$, with respect to the open cover $\mathfrak{U} = \{ U_i \}_{i \in I}$, where $\psi_{ij}$ is defined by (\ref{E:cocycle}). One can extend the isomorphisms $\varphi_i$ from (\ref{E:Isom7}) to linear isomorphisms $K^n \to K^n$ (without changing notations), i.e. view $\phi_i$ as an element of $\mathrm{GL}_n(K) = G(K)$. As in the proof of Proposition \ref{P:DoubleC}, we pick a map $\phi \colon V \to I$ such that $\mathfrak{p}_v \in U_{\phi(v)}$ for all $v \in V$. Furthermore, fix $i_0 \in I$ and for the corresponding element take $\varphi_{i_0}$. Then
$$
\psi_{i i_0} \varphi_{i_0} = \varphi_i \ \ \text{for all} \ \ i \in I.
$$
Then $\theta$ takes $[\psi(M)]$ to the double coset $G(K) g G(\mathbf{A}\!^{\infty})$ where $g_v = \varphi_{\phi(v)}^{-1}$ for all $v \in V$. Implementing the construction described in the proof of Proposition \ref{P:Ap2}, we see that $\beta$ take this double cosets to the isomorphism class of the module
$$
\bigcap_{\mathfrak{p} \in \mathrm{P}} \varphi_{\phi_{v_{\mathfrak{p}}}}^{-1}(R_{\mathfrak{p}}^n) = \bigcap_{\mathfrak{p} \in \mathrm{P}} M_{\mathfrak{p}} = M,
$$
proving the commutativity of (\ref{E:Diagr}).
\end{proof}

\vskip2mm

\addtocounter{thm}{1}

\noindent {\bf Remark 6.5.} It follows from Proposition \ref{P:Ap3} that contrary to the result of Harder \cite[2.3]{Harder} that the map $f \colon \check{H}^1(X , G) \to G(\mathbf{A}\!^{\infty}) \backslash G(\mathbf{A}) / G(K)$ is bijective for any affine flat group scheme $G$ of finite type over a Dedekind ring $R$, this map may not be surjective even for $G = \mathrm{GL}_n$ for all $n$ over more general rings. Indeed, by the proposition its surjectivity for $\mathrm{GL}_n$ for all $n$ would imply that finitely generated reflexive $R$-modules are all projective. On the other hand, it is well-known that the second syzygy (see \cite[vol. II, \S5.1]{Rowen} for the definition) of any finitely generated $R$-module is reflexive, and, conversely, every finitely generated reflexive module can be realized as the second syzygy of some finitely  generated $R$-module (cf. \cite{Treg}, p. 445). So, in our situation, we would obtain that the {\it global dimension} of $R$ is $\leq 2$ (see \cite[Ch. 2, \S5]{Lam} and \cite[Ch. IV]{Serre-LocAlg} for relevant definitions). We also recall that for a regular noetherian ring $R$, the global dimension coincides with Krull dimension \cite[Ch. 2, Theorem 5.94]{Lam}. Thus, if we take $R$ to be a noetherian regular domain of Krull dimension $\geq 3$ (e.g. $R = \Q[x, y, z]$ or $\Z[x , y]$), then the above map $f$ is {\it not} surjective for $G = \mathrm{GL}_n$ for some $n$.

\vskip2mm

This remark contrasts with the following statement.
\begin{thm}\label{T:Ap2-1}
{\rm (cf. \cite[Proposition 2]{Sam})} Let $R$ be a regular integral domain of Krull dimension $\leqslant 2$. Then every reflexive $R$-lattice is a projective $R$-module. In other words, {\rm $\texttt{Refl}_n(R) = \texttt{Proj}_n(R)$} for all $n \geqslant 1$.
\end{thm}

A proof quickly follows, for example, from \cite[Corollary 6, p. 78]{Serre-LocAlg} which implies that given a reflexive $R$-lattice $M$, for every prime ideal $\mathfrak{p}$ of $R$, the localization $M_{\mathfrak{p}}$ is a free $R_{\mathfrak{p}}$-module, making $M$ projective. We note that this result for $R = \Z_p[\![T]\!]$ is fundamental for Iwasawa theory, and we refer the reader to \cite{Ouyang} and \cite{Sharifi} for contemporary accounts that contain the proof of Theorem \ref{T:Ap2-1} and further results on the structure of modules over this ring.

\vskip2mm

\begin{cor}\label{C:Ap1}
Let $G = \mathrm{GL}_n$, let $R$ be a regular domain of Krull dimension $\leqslant 2$, and let $X = \mathrm{Spec}\: R$. Then then map
$$
f \colon \check{H}^1(X , G) \longrightarrow G(\mathbf{A}\!^{\infty}) \backslash G(\mathbf{A}) / G(K),
$$
constructed in Proposition \ref{P:Ap1}, is a bijection.
\end{cor}

This result prompts the following

\addtocounter{thm}{1}

\vskip2mm

\noindent {\bf Question 6.8.} {\it Is the map $f \colon \check{H}^1(X , G) \to  G(\mathbf{A}\!^{\infty}) \backslash G(\mathbf{A}) / G(K)$ bijective for any affine flat group scheme of finite type over a regular domain $R$ of Krull dimension $\leqslant 2$?}

\vskip2mm

We note that the proof of the above corollary hinges on the interpretation of $\check{H}^1(X , G)$ and $G(K) \backslash G(\mathbf{A}) / G(\mathbf{A}\!^{\infty})$ for $G = \mathrm{GL}_n$ in terms of projective and reflexive modules. It would be interesting to obtain similar interpretations for some other groups (e.g., it is known in the number-theoretic situation that when $G$ is the orthogonal group of a quadratic form, the double cosets $G(\mathbf{A}\!^{\infty}) \backslash G(\mathbf{A}) / G(K)$ are in bijection with the classes in the genus of the quadratic form, cf. \cite[Proposition 8.4]{PlRa}.)

\vskip2mm

We will now quote two famous results about projective modules and derive some consequences for $\check{H}^1(X , G)$ and double cosets for $G = \mathrm{GL}_n$. 

\begin{thm}\label{T:Serre1}
{\rm (Serre \cite{Serre})} Let $R$ be a commutative noetherian ring of Krull dimension $d < \infty$. If $P$ is a finitely generated projective $R$-module of constant rank $r \geqslant d$, then $P = Q \oplus F$, where $F$ is free and $Q$ has rank at most $d$.
\end{thm}

\begin{thm}\label{T:Bass1}
{\rm (Bass \cite{Bass})} Let $R$ be a commutative noetherian ring of Krull dimension $d < \infty$, and let $P$ and $Q$ be  finitely generated projective $R$-modules of constant rank $r > d$. If $P \oplus F \simeq Q \oplus F$ with $F$ free and finitely generated, then $P \simeq Q$.
\end{thm}

\vskip2mm

Combining these results with Proposition \ref{P:Ap1}, we obtain the following.
\begin{cor}\label{C:Ap3}
Let $R$ be a noetherian integrally closed domain of Krull dimension $d < \infty$, and let $X = \mathrm{Spec}\: R$. For $1 \leqslant m \leqslant n$, we consider the map
$$
\lambda_{m , n} \colon \check{H}^1(X  , \mathrm{GL}_m) \longrightarrow \check{H}^1(X , \mathrm{GL}_n)
$$
induced by the natural embedding $\mathrm{GL}_m \hookrightarrow \mathrm{GL}_n$. Then $\lambda_{m , n}$ is surjective if $m \geqslant d$ and is
bijective if $m > d$.
\end{cor}

Indeed, in terms of the identifications provided by Proposition \ref{P:Ap1}, the map $\lambda_{m , n}$ corresponds to the map $\texttt{Proj}_m(R) \to
\texttt{Proj}_n(R)$ given by $[P] \mapsto [P \oplus R^{n-m}]$.

\vskip1mm

In conjunction with Corollary \ref{C:Ap1}, this statement yields
\begin{cor}\label{C:Ap4}
Let $R$ be a noetherian regular domain of Krull dimension $\leqslant 2$, and let $X = \mathrm{Spec}\: R$. For $1 \leqslant m \leqslant n$, we consider the map
$$
\mu_{m , n} \colon \mathrm{GL}_m(\mathbf{A}^{\infty}) \backslash \mathrm{GL}_m(\mathbf{A}) / \mathrm{GL}_m(K) \longrightarrow
\mathrm{GL}_n(\mathbf{A}\!^{\infty}) \backslash \mathrm{GL}_n(\mathbf{A}) / \mathrm{GL}_n(K)
$$
induced by the natural embedding $\mathrm{GL}_m \hookrightarrow \mathrm{GL}_n$. Then $\mu_{m , n}$ is surjective for $m \geqslant 2$ and bijective for $m > 2$.
\end{cor}

It would be interesting to determine if $\mu_{m , n}$ is surjective (resp., bijective) whenever $m \geqslant d$ (resp. $m > d$) for any noetherian regular domain of Krull dimension $d < \infty$. Among other things, this would help to see if the restrictions on $R$ from Remark 6.5 for $f$ to be surjective would also be necessary for Condition (T) to hold. More precisely, suppose that for a given $R$, Condition (T) holds {\it uniformly} for all $\mathrm{GL}_n$, i.e. there exists $V' \subset V$ with finite complement such that $\vert \mathrm{Cl}(\mathrm{GL}_n, K, V') \vert = 1$ for all $n$. Without loss of generality, we may assume that $V' = V \setminus V(a)$ for some nonzero $a \in R$, and then by Proposition \ref{P:Ap2} our assumption  would imply that $\vert \texttt{Refl}_n(R_a) \vert = 1$ for all $n$, i.e. every finitely generated reflexive $R_a$-module is free. As we discussed in Remark 6.5, for $R$ regular noetherian this is possible only if its Krull dimension is $\leqslant 2$. Thus, if $R$ is a regular noetherian domain of Krull dimension $d \geqslant 3$, then Condition (T) cannot hold uniformly. However, if we knew that the maps $\mu_{m , n}$ are bijections for all $n \geqslant m  > d$, we could say that Condition (T) actually fails for some $\mathrm{GL}_t$ with $t \leqslant d + 1$. In any case, we do expect Condition (T) to hold if $R$ is a regular domain of Krull dimension $\leqslant 2$ which is a finitely generated algebra over $\Z$ or a finitely generated field (cf. Corollary \ref{C:Ap5} below).

\vskip2mm

For a commutative ring $R$, we let $K_0(R)$ denote its Grothendieck group (cf. \cite{Weibel}).

\begin{lemma}\label{L:Ap1}
Let $R$ be a regular noetherian integral domain of Krull dimension $d < \infty$ such that the group $K_0(R)$ is finitely generated. Then there exists a nonzero $a \in R$ such that over the localization $R_a$, every finitely generated projective module of rank $> d$ is free.
\end{lemma}
\begin{proof}
Pick finitely generated projective $R$-modules $P_1, \ldots , P_r$ so that the corresponding classes $$[P_1], \ldots , [P_r] \in K_0(R)$$ generate this group. Then there exists a nonzero $a \in R$ such that all localization $(P_1)_a, \ldots , (P_r)_a$ are free $R_a$-modules. Since $R$ is regular, the map $K_0(R) \to K_0(R_a)$ is surjective (cf. \cite[Ch. II, Application 6.4.1, p. 132]{Weibel}) implying that $K_0(R_a)$ is infinite cyclic with generator $[R_a]$. This means that for any finitely generated projective $R_a$-module $P$, there exist $\ell , m \geqslant 0$ such that $P \oplus (R_a)^\ell \simeq (R_a)^m$. Then if $P$ has rank $> d$, it is free by Theorem \ref{T:Bass1} since the Krull dimension of $R_a$ is $\leqslant d$.
\end{proof}

\vskip.5mm

One may be able to use this lemma to derive some of the required finiteness properties from the following old and notoriously difficult conjecture.

\vskip2mm

\noindent {\bf Conjecture \ref{S:GLn}.14.} (Bass \cite[\S9]{Bass-Conj}) {\it Let $R$ be a finitely generated $\Z$-algebra that is a regular ring. Then the group $K_0(R)$ is finitely generated.}

\vskip2mm

(We note that Bass actually conjectured finite generation of all groups $K_n(R)$ $(n \geqslant 0)$ for such $R$.)

\addtocounter{thm}{1}

\vskip2mm

We would like to end this section with some consequences of Bass's conjecture in the context of our finiteness properties.
\begin{cor}\label{C:Ap5}
Let $R$ be a regular domain of Krull dimension $d < \infty$ which is a finitely generated $\Z$-algebra, $X = \mathrm{Spec}\: R$ and $G = \mathrm{GL}_n$ with $n > d$. If Bass's Conjecture \ref{S:GLn}.14 is true, then there exists an open $U \subset X$ such that $\check{H}^1(U , G) = 1$, i.e. Conjecture 5.2 holds.
\end{cor}

Indeed, assuming the truth of Bass's conjecture, we will have that $K_0(R)$ is finitely generated.  Since $n > d$, we can apply Lemma \ref{L:Ap1} to conclude that there exists a nonzero $a \in R$ such that $\texttt{Proj}_n(R_a)$ reduces to a single element. Then Proposition \ref{P:Ap1} implies that our claim holds for the principal open subset $U \subset X$ defined by $a$.

\vskip1mm

\begin{cor}\label{C:Ap6}
Let $R$ be a regular noetherian domain of Krull dimension $\leqslant 2$ which is a finitely generated $\Z$-algebra, $V$ be the set of discrete valuations of the fraction field $K$ associated with height one primes of $R$, and $G = \mathrm{GL}_n$ with $n \geqslant 3$. If Bass's Conjecture \ref{S:GLn}.14 is true, then the adele group of $G$ associated with $V$ satisfies Condition {\rm (T)}.
\end{cor}

Indeed, the proof of the previous corollary shows that there exists a nonzero $a \in R$ such that for $U = \mathrm{Spec}\: R_a$ we have $\check{H}^1(U , G) = 1$. On the other hand, since the Krull dimension of $R$ is $\leqslant 2$, according to Corollary \ref{C:Ap1} there is a bijection between $\check{H}^1(U , G)$ and the double cosets $G(\mathbf{A}\!^{\infty}(K , V')) \backslash G(\mathbf{A}(K , V')) / G(K)$ of the adele group associated with $V' = V \setminus V(a)$ where $V(a) = \{ v \in V \: \vert \: v(a) \neq 0 \}$ (finite set).

\section{Descent}\label{S:Descent}

As we have seen in the previous section, in certain situations condition (T) for $\mathrm{GL}_n$ $(n \geqslant 3)$ can be derived from Bass's conjecture. Applications described in the main body of the paper, however, require this condition for groups of the form $\mathrm{GL}_{\ell , D}$, where $D$ is a finite-dimensional division algebra. In this section, we will describe a descent procedure that, under certain additional assumptions, enables one to derive Condition (T) for $\mathrm{GL}_{\ell , D}$ from its validity for $\mathrm{GL}_n$.

\vskip2mm

We first review standard Galois descent. Let $L/K$ be a finite Galois extension of degree $d$ with Galois group $\mathscr{G}$, and let $W$ be a vector space over $L$ endowed with a semi-linear action of $\mathscr{G}$. This means that the action of every $\sigma \in \mathscr{G}$ satisfies
$$
\sigma(w_1 + w_2) = \sigma(w_1) + \sigma(w_2) \ \ \text{and} \ \ \sigma(a w) = \sigma(a) \sigma(w)
$$
for all $w, w_1, w_2 \in W$ and $a \in L$.

\vskip1mm

Now, we assume that $K$ is the fraction field of an integrally closed noetherian domain $R$, and let $R'$ denote the integral closure of $R$ in $L$ (which is automatically noetherian). Assume further that

\vskip2mm

(a) $R'$ is a free $R$-module;

\vskip1mm

(b)\footnotemark \ \parbox[t]{15cm}{the discriminant of the trace form $$L \times L \to K, \ \  (x , y) \mapsto \mathrm{Tr}_{L/K}(xy),$$ in some (equivalently, any) $R$-basis of $R'$ is a unit in $R$.}

\footnotetext{We note that this is equivalent to the fact that the map $R' \otimes_{R} R' \to \prod_{\sigma \in \mathscr{G}} R'$, $a \otimes b \mapsto (a\sigma(b))_{\sigma \in \mathscr{G}}$, is an isomorphism, i.e. $R'/R$ is a Galois extension of rings in the usual sense.}

\vskip2mm

\noindent We note that one can ensure that both conditions hold by replacing $R$ by its localization with respect to an appropriate nonzero element
$a \in R$.

\vskip2mm

\begin{lemma}\label{L:Descent}
Keep the above notations and assumptions, and let $M \subset W$ be a $\mathscr{G}$-invariant $R'$-submodule. Then for the $R$-module $M_0 = M^{\mathscr{G}}$ of fixed elements, the canonical map $M_0 \otimes_{R} R' \to M$ is an isomorphism.
\end{lemma}
\begin{proof}
We only need to show that $R'M_0 = M$. Let $\mathscr{G} = \{ \sigma_1, \ldots , \sigma_d \}$ and let $a_1, \ldots , a_d$ be an $R$-basis of $R'$. Consider the matrix $A = (\sigma_i(a_j)) \in \mathrm{M}_d(R')$. It is easy to check that $A^tA = (\mathrm{Tr}_{L/K}(a_ia_j))$, so it follows from condition (b) above that $A \in \mathrm{GL}_d(R')$. Now, for $w \in M$ and any $j = 1, \ldots d$, we define
$$
\lambda_j(w) = \sum_{i = 1}^d \sigma_i(a_j w) = \sum_{i = 1}^d \sigma_i(a_j) \sigma_i(w).
$$
Clearly, $\lambda_j(w) \in M_0$. We have the following ``matrix" relation
$$
\left( \begin{array}{c} \lambda_1(w) \\ \vdots \\ \lambda_d(w)   \end{array} \right) \, =  \, A^t \cdot \left( \begin{array}{c} \sigma_1(w) \\ \vdots \\\sigma_d(w)     \end{array}  \right), \ \ \ \text{whence} \ \ \left( \begin{array}{c} \sigma_1(w) \\ \vdots \\\sigma_d(w)     \end{array}  \right) \, = \, (A^t)^{-1} \cdot \left( \begin{array}{c} \lambda_1(w) \\ \vdots \\ \lambda_d(w)   \end{array} \right).
$$
This implies that $\sigma_i(w) \in R' M_0$ for all $i = 1, \ldots , d$. In particular, $w \in R' M_0$, as required. \end{proof}

\vskip2mm

Now, let $C$ be a finite-dimensional algebra over $K$, and let $\mathscr{C} \subset C$ be an $R$-order (i.e., a subring with identity that is an $R$-submodule and contains a $K$-basis of $C$). Then $\mathscr{G}$ acts on $\mathscr{C} \otimes_{R} R'$  through the second factor.
\begin{prop}\label{P:Descent}
Assume that

\vskip2mm

\noindent {\rm (*)} \parbox[t]{15.5cm}{for a right $\mathscr{C}$-module $M$, the fact that $M^d \simeq \mathscr{C}^d$ implies that $M \simeq \mathscr{C}$.}

\vskip2mm

\noindent Then $H^1(\mathscr{G} , (\mathscr{C} \otimes_{R} R')^{\times}) = 1$.
\end{prop}
\begin{proof}
We will view $\mathscr{C}$ and $\mathscr{C} \otimes_{R} R'$ as {\it right} $\mathscr{C}$-modules, noting an isomorphism $\mathscr{C} \otimes_{R} R' \simeq \mathscr{C}^d$. Now, given a cocycle $\xi \in Z^1(\mathscr{G} , (\mathscr{C} \otimes_{R} R')^{\times})$, we define a new (twisted) action of $\mathscr{G}$ on $\mathscr{C} \otimes_{R} R'$ by
$$
\sigma^*(a) = \xi(\sigma) \cdot \sigma(a) \ \ \text{for} \ \ \sigma \in \mathscr{G}, \ a \in \mathscr{C} \otimes_{R} R'
$$
(the dot denotes the product in $\mathscr{C} \otimes_{R} R'$). This is a semi-linear action that commutes with the right multiplication by elements of $\mathscr{C}$. Let $M_0$ denote the set of fixed elements under this action; clearly $M_0$ is a right $\mathscr{C}$-module. It follows from
Lemma \ref{L:Descent} that the canonical map $M_0 \otimes_{R} R' \to \mathscr{C} \otimes_{R} R'$ is an isomorphism of $R'$-modules.
Since $\mathscr{C}$ and $R'$ commute inside $\mathscr{C} \otimes_{R} R'$, this map respects the right $\mathscr{C}$-module structures. But as $\mathscr{C}$-modules,
$$
M_0 \otimes_{R} R' \simeq M_0^d \ \ \text{and} \ \ \mathscr{C} \otimes_{R} R' \simeq \mathscr{C}^d.
$$
Thus, $M_0^d \simeq \mathscr{C}^d$ as right $\mathscr{C}$-modules, which in view of (*) implies the existence of a $\mathscr{C}$-isomorphism $f \colon \mathscr{C} \to M_0$. Set $c = f(1)$; then $f(x) = c x$ for all $x \in \mathscr{C}$. Since the extension $$\mathscr{C} \otimes_{R} R' \to M_0 \otimes_{R} R' \simeq \mathscr{C} \otimes_{R} R',$$ which is still given by multiplication by $c$, remains an isomorphism, we conclude that $c \in (\mathscr{C} \otimes_{R} R')^{\times}$. On the other hand, the condition $\sigma^*(c) = c$ means that
$$
c^{-1} \xi(\sigma) \sigma(c) = 1 \ \  \text{for all} \ \  \sigma \in \mathscr{G},
$$
i.e. $\xi$ is equivalent to the trivial cocycle.
\end{proof}

\vskip1mm

Now, let $G = \mathrm{GL}_{1 , C}$ be the algebraic $K$-group associated with $C$; recall that $G(B) = (C \otimes_K B)^{\times}$ for any $K$-algebra $B$. (If $C$ is a central simple algebra and $C = \mathrm{M}_{\ell}(D)$, where $D$ is a central division $K$-algebra, then $G$ can also be described as $\mathrm{GL}_{\ell , D}$.) Next, assume that $\mathscr{C} \subset C$ is a {\it free} $R$-order. The regular representation $C \hookrightarrow \mathrm{M}_m(K)$, where $m = \dim_K C$, written in an $R$-basis of $\mathscr{C}$, gives rise to a faithful $K$-representation $G \hookrightarrow \mathrm{GL}_m$ such that the corresponding groups $G(R)$ and $G(R')$ can be identified with $\mathscr{C}^{\times}$ and $(\mathscr{C} \otimes_{R} R')^{\times}$, respectively. We let $V$ denote the set of discrete valuations associated with height one primes of $R$, and let $V'$ be the set of all extensions of these valuations to $L$ --- this is precisely the set of discrete valuations of $L$ associated with height one primes of $R'$. These sets will remain fixed throughout the rest of the section, so the adele groups $G(\mathbf{A}(K , V))$ and $G(\mathbf{A}(L , V'))$ will be denoted simply by $G(\mathbf{A}(K))$ and $G(\mathbf{A}(L))$.
\begin{cor}\label{C:Descent1}
Assume that condition {\rm (*)} of Proposition \ref{P:Descent} holds. If $G(\mathbf{A}(L)) = G(\mathbf{A}\!^{\infty}(L)) G(L),$ then $G(\mathbf{A}(K)) = G(\mathbf{A}\!^{\infty}(K)) G(K)$.
\end{cor}
Indeed, as we discussed in \S\ref{S:T}, in our situation there exists an injective map
$$
\lambda \colon G(\mathbf{A}\!^{\infty}(K)) \backslash G(\mathbf{A}(K)) / G(K) \longrightarrow H^1(\mathscr{G} , G(R')).
$$
But $G(R') = (\mathscr{C} \otimes_{R} R')^{\times}$, so due to condition (*), we have $H^1(\mathscr{G} , G(R')) = 1$ by Proposition \ref{P:Descent}, and our claim follows. \hfill $\Box$

\addtocounter{thm}{1}

\vskip2mm

\noindent {\bf Remark 7.4.} In \S\ref{S:GLn}, we used Bass's Cancellation Theorem \ref{T:Bass1} to derive some finiteness results for $G = \mathrm{GL}_n$ over $R$ for sufficiently large $n$ from Bass's Conjecture \ref{S:GLn}.14 about finite generation of $K_0(R)$. While noncommutative versions of Theorems  \ref{T:Serre1} and \ref{T:Bass1} are available - see \cite[vol. II, Corollaries 5.1.60 and 5.1.61]{Rowen}, they are rather technical. So, we would like to point out that the stable version of descent can be implemented on the basis of information about $K_0$ with relatively little effort. More precisely, we consider the group $K_0(\mathscr{C})$ defined in terms of {\it right} projective $\mathscr{C}$-module, and assume that it  has no $d$-torsion (where $d = [L : K]$), i.e. $d x = 0$ for $x \in K_0(\mathscr{C})$ implies $x = 0$, which, generally speaking, is a weaker condition than (*) in Proposition \ref{P:Descent}. Then the proof of this proposition yields the following fact: For $t \geqslant 1$, we consider the canonical embedding $\tau_t \colon (\mathscr{C} \otimes_{R} R')^{\times} \hookrightarrow \mathrm{GL}_t(\mathscr{C} \otimes_{R} R')$ given by (\ref{E:Can-Emb}). Then for every $\xi \in Z^1(\mathscr{G} , (\mathscr{C} \otimes_{R} R')^{\times})$, there exists $t \geqslant 1$ (depending on $\xi$) such that the image $\xi_t$ of $\xi$ in $Z^1(\mathscr{G} , \mathrm{GL}_t(\mathscr{C} \otimes_{R} R'))$ under $\tau_t$ is equivalent to the trivial cocycle. Indeed, in the proof we introduced the twisted Galois action which led to  a right $\mathscr{C}$-module $M_0$ such that $M_0^d \simeq \mathscr{C}^d$. Using (*), we concluded that $M_0 \simeq \mathscr{C}$, which immediately implied that $\xi$ is equivalent to the trivial cocycle. The absence of $d$-torsion in $K_0(\mathscr{C})$ only tells us that $[M_0] = [\mathscr{C}]$ in $K_0(\mathscr{C})$, i.e. $M_0 \oplus \mathscr{C}^{t-1} \simeq \mathscr{C}^t$ for some $t \geqslant 1$. Then twisting $\mathscr{C}_t := {\rm M}_t(\mathscr{C})$ using $\xi_t$ we obtain $M_{0 t} := (M_0 \oplus \mathscr{C}^{t-1})^t$ which as a right $\mathscr{C}_t$-module is isomorphic to $(\mathscr{C}^t)^t$, i.e. to $\mathscr{C}_t$. But an isomorphism $M_{0 t} \simeq \mathscr{C}_t$ implies that $\xi_t$ is equivalent to the trivial cocycle, as claimed.

\vskip2mm

Now, if it is known that for $G_t = \mathrm{GL}_{t , C}$  we have $G_t(\mathbf{A}(L)) = G_t(\mathbf{A}\!^{\infty}(L)) G_t(L)$ for all sufficiently large $t$, then the proof of Corollary \ref{C:Descent1} in conjunction with the above discussion shows that for every $g \in G_1(\mathbf{A}(K))$, there exists $t \geqslant 1$ (depending on $g$) such that $\tau_t(g) \in G_t(\mathbf{A}(K))$ lies in the principal class $G_t(\mathbf{A}\!^{\infty}(K)) G_t(K)$.

\vskip2mm

Let now $C$ be a finite-dimensional {\it central simple} algebra over $K$, and $G = \mathrm{GL}_{1 , C}$ be the corresponding reductive algebraic group. As above, we fix a free $R$-order $\mathscr{C} \subset C$, and consider a faithful representation of $G$ associated with a basis of this order.

We will now take for $L$ a finite Galois extension of $K$ that splits $C$, i.e. there exists an isomorphism of $L$-algebras $\theta \colon C \otimes_K L \to \mathrm{M}_n(L)$.  Let us make the following assumption

\vskip2mm

 (c) \ $\theta(\mathscr{C}) \subset \mathrm{M}_n(R')$ \ and \ $R' \theta(\mathscr{C}) = \mathrm{M}_n(R')$.

\vskip2mm

\noindent This condition implies that for any $w \in V'$, the map $\theta$ identifies $(\mathscr{C} \otimes_{R} R') \otimes_{R'} \mathscr{O}_{L , w}$ with $\mathrm{M}_n(\mathscr{O}_{L , w})$, and hence $G(\mathscr{O}_{L , w})$ with $\mathrm{GL}_n(\mathscr{O}_{L , w})$. Consequently, $\theta$ enables us to identify the adele group $G(\mathbf{A}(L))$
with $\mathrm{GL}_n(\mathbf{A}(L))$ so that the subgroups $G(\mathbf{A}\!^{\infty}(L))$ and $G(L)$ get identified with $\mathrm{GL}_n(\mathbf{A}\!^{\infty}(L))$ and
$\mathrm{GL}_n(L)$, respectively. In particular, there is a bijection
\begin{equation}\label{E:Bijection}
G(\mathbf{A}\!^{\infty}(L)) \backslash G(\mathbf{A}(L)) / G(L) \longrightarrow \mathrm{GL}_n(\mathbf{A}\!^{\infty}(L)) \backslash \mathrm{GL}_n(\mathbf{A}(L)) / \mathrm{GL}_n(L).
\end{equation}
We now observe that since our goal is to develop an approach for verifying Condition (T), we can freely drop from $V$ any finite number of valuations, in other words, to replace $R$ with its localization with respect to a nonzero $a \in R$. On the other hand, by applying such localization we can ensure that conditions (a), (b) and (c) hold. Moreover, the results of \S\ref{S:GLn} imply that if $R$ is a finitely generated $\Z$-algebra of Krull dimension $\leqslant 2$ and we assume the truth of Bass's conjecture, then by localizing it further we may suppose that
$$
\mathrm{GL}_m(\mathbf{A}(L)) = \mathrm{GL}_m(\mathbf{A}\!^{\infty}(L)) \mathrm{GL}_m(L) \ \ \text{for all} \ \ m \geqslant 3.
$$
In view of the bijection (\ref{E:Bijection}), this means that if $C$ is of degree $n \geqslant 3$ (e.g. $C = \mathrm{M}_{\ell}(D)$ where $\ell \geqslant 2$ and $D$ is a quaternion algebra) then $G(\mathbf{A}(L)) = G(\mathbf{A}\!^{\infty}(L)) G(L)$. Combining this with Corollary \ref{C:Descent1}, we obtain the following.
\begin{thm}\label{T:Descent}
Let $K$ be the fraction field of a finitely generated integrally closed $\Z$-algebra $R$ of Krull dimension $\leqslant 2$, and assume that Bass's conjecture holds. Suppose that $C$ is a central simple $K$-algebra of degree $n \geqslant 3$ that admits a free $R$-order $\mathscr{C}$ for which condition $(*)$ of Proposition \ref{P:Descent} holds, with $d$ being the degree of some Galois extension $L/K$ that splits $C$. Then Condition {\rm (T)} holds for the $K$-group $G = \mathrm{GL}_{1 , C}$.
\end{thm}

\vskip2mm

Now, keep the above assumptions {\it except} condition (*) of Proposition \ref{P:Descent} and replace the latter with the assumption that $K_0(\mathscr{C})$ has no $d$-torsion. Then using Remark 7.4 instead Corollary \ref{C:Descent1}  in the proof of the above theorem, one obtains a ``stable" version of condition (T): For $t \geqslant 1$, let $G_t = \mathrm{GL}_{t , C}$ and let $\tau_t \colon G_1 \hookrightarrow G_t$ be the canonical embedding. Then by dropping from $V$ finitely many valuations, one can ensure that for any $g \in G_1(\mathbf{A}(K))$, there exists $t \geqslant 1$ (depending on $g$) such that $\tau_t(g) \in G_t(\mathbf{A}(K))$ lies in the principal class $G_t(\mathbf{A}\!^{\infty}(K)) G_t(K)$.
As we already pointed out in \S\ref{S:G2}, the stable version is sufficient for the relevant results in that section, cf. Remarks 3.6  and 3.9.

\vskip2mm

\noindent {\small  {\bf Acknowledgements.} The first author was supported by an NSERC research grant. During the preparation of the paper, the second author visited Princeton University and the Institute for Advanced Study on a Simons Fellowship; the hospitality of both institutions and the generous support of the Simons Foundation are thankfully acknowledged. The third author was partially supported by an AMS-Simons Travel Grant. We would like to thank the anonymous referee for offering a number of corrections and valuable suggestions. We are also grateful to Louis Rowen for useful discussions.}

\bibliographystyle{amsplain}

\begin{thebibliography}{100}

\bibitem{ANT} {\it Algebraic Number Theory. Second edition,} ed. J.W.S.~Cassels and A.~Fr\"ohlich, London Math. Soc., 2010.


\bibitem{Bass} H.~Bass, {\it K-theory and stable algebra,} Publ. math. IHES {\bf 22}(1964), 5-60.

\bibitem{Bass-Conj} H. Bass, {\it Some problems in classical algebraic K-theory}, Algebraic K- theory, II: �Classical�
algebraic K-theory and connections with arithmetic (Proc. Conf., Battelle Memorial Inst.,
Seattle, Wash., 1972), pp. 3�73. Lect. Notes Math. {\bf 342}, 1973.

\bibitem{Borel} A.~Borel, {\it Some finiteness properties of adele groups over number fields,} Publ. math. IHES {\bf 16}(1963), 5-30.

\bibitem{Bour-CA} N.~Bourbaki, {\it Commutative Algebra. Chapters 1-7}, Springer, 1989.

\bibitem{CG} V.I.~Chernousov, V.I.~Guletski\i, {2-torsion of the Brauer group of an elliptic curve: generators and relations,} Doc. math. 2001.
Extra vol., 85-120.

\bibitem{CRR1} V.I.~Chernousov, A.S.~Rapinchuk, and I.A.~Rapinchuk, {\it The genus of a division algebra and the unramified Brauer group,} Bull.
Math. Sci. {\bf 3}(2013), no. 2, 211-240.

\bibitem{CRR1a} V.I.~Chernousov, A.S.~Rapinchuk, and I.A.~Rapinchuk, {\it Division algebras with the same maximal subfields}, Russian Math. Surveys
{\bf 70}:1 (2015), 83-112.


\bibitem{CRR2} V.I.~Chernousov, A.S.~Rapinchuk, and I.A.~Rapinchuk, {\it On the size of the genus of a division algebra,}  Tr. Mat. Inst. Steklova {\bf 292}(2016), Algebra, Geometriya i Teoriya Chisel, 69-99.

\bibitem{CRR3a} V.I.~Chernousov, A.S.~Rapinchuk, and I.A.~Rapinchuk, {\it On some finiteness properties of algebraic groups over finitely
generated fields}, C.R. Acad. Sci. Paris, Ser. I {\bf 354}(2016), 869-873.

\bibitem{CRR3} V.I.~Chernousov, A.S.~Rapinchuk, and I.A.~Rapinchuk, {\it Spinor groups with good reduction,} to appear in Compos. math., arXiv:1701.08062.

\bibitem{DeFa} B.~Farb, R.K.~Dennis, {\it Noncommutative Algebra}, GTM 144, Springer, 1993.


\bibitem{Gille} P.~Gille, T.~Szamuely, {\it Central simple algebras and Galois cohomology,}
Cambridge Studies in Advanced Mathematics, {\bf 101}. Cambridge University Press, Cambridge, 2006.

\bibitem{Harder} G.~Harder, {\it Halbeinfache Gruppenschemata \"{u}ber Dedekindringen,} Invent. math. {\bf 4}(1967), 165-191.

\bibitem{Hart} R.~Hartshorne, {\it Algebraic Geometry,} GTM 52, Springer, 1977.

\bibitem{Kahn} B.~Kahn, {\it Sur le groupe des classes d'un sch\'ema arithm\'etique,} Bull. Soc. math. France {\bf 134}(2006), 395-415.

\bibitem{Lam} T.Y.~Lam, {\it Lectures on Modules and Rings}, GTM 189, Springer, 1999.

\bibitem{MS} A.S.~Merkurjev, A.A.~Suslin, {\it $K$-cohomology of Severi-Brauer varieties and the norm residue homomorphism,} Izv. Akad. Nauk SSSR Ser. Mat. {\bf 46}(1982), no. 5, 1011-1046.


\bibitem{Meyer} J.S.~Meyer, {\it Division algebras with infinite genus,} Bull. Lond. Math. Soc. {\bf 46}(2014), no. 3, 463-468.


\bibitem{Milne} J.S.~Milne, {\it Lectures on Etale Cohomology,} available at: \verb"http://www.jmilne.org/math/CourseNotes/lec.html"

\bibitem{Ouyang} Y.~Ouyang, {\it Introduction to Iwasawa Theory}, available at: \verb"https://www.math.unipd.it/~algant/iwasawa.pdf"

\bibitem{PlRa} V.P.~Platonov, A.S.~Rapinchuk, {\it Algebraic Groups and Number Theory,} Academic Press, 1993.

\bibitem{R-SA} A.S.~Rapinchuk, {\it Strong approximation for algebraic groups,} Thin groups and superstrong approximation, 269-298, Math. Sci. Res. Inst. Publ., {\bf 61}, Cambridge Univ. Press, Cambridge, 2014.

\bibitem{RR} A.S.~Rapinchuk, I.A.~Rapinchuk, {\it On division algebras having the same maximal subfields,} Manuscr. math. {\bf 132}(2010), 273-293.

\bibitem{Rohl} J.~Rohlfs, {\it Arithmetische definierte Gruppen mit Galois-operation,} Invent. math. {\bf 4}(1978), no. 2, 185-205.

\bibitem{Rowen} L.H.~Rowen, {\it Ring Theory} (in two volumes), Academic Press, 1988. 

\bibitem{Sam} P.~Samuel, {\it Anneaux gradues factoriels,} Bull. Soc. Math. France {\bf 92}(1964), 237-249.

\bibitem{Samuel} P.~Samuel, {\it \`A propos du th\'eor\`eme des unit\'es,} Bull. Sci. Math. (2) {\bf 90}(1966), 89-96.

\bibitem{Salt} D.J.~Saltman, {\it Lectures on Division Algebras}, CMBS Regional Conference Series in Mathematics, No. 94, AMS, 1999.

\bibitem{Serre} J.-P.~Serre, {\it Modules projectifs et espaces fibr\'es \`a fibre vectorielle,} Sem. Dubreil-Pisot no. 23, Paris, 1957/1958.

\bibitem{Serre-LF} J.-P.~Serre, {\it Local Fields}, GTM 67, Springer, 1979.

\bibitem{Serre-LocAlg} J.-P.~Serre, {\it Local Algebra}, Springer, 2000.

\bibitem{Serre-GC} J.-P.~Serre, {\it Galois Cohomology}, Springer, 2001.

\bibitem{Sharifi} R.~Sharifi, {\it Iwasawa theory}, available at: \verb"http://math.ucla.edu/~sharifi/iwasawa.pdf"



\bibitem{Tikhon} S.V.~Tikhonov, {\it Division algebras of prime degree with infinite genus},  Tr. Mat. Inst. Steklova {\bf 292}(2016), Algebra, Geometriya i Teoriya Chisel, 264-267.


\bibitem{Treg} R.~Treger, {\it Reflexive modules}, J.~Algebra {\bf 54}(1978), 444-446.

\bibitem{Wads} A.R.~Wadsworth, {\it Valuation theory on finite dimensional division algebras}, Valuation theory and its applications, Vol. I (Saskatoon, SK, 1999), 385-449, Fields Inst. Commun., {\bf 32}, AMS, 2002.

\bibitem{Weibel} C.~Weibel, {\it The $K$-book. An Introduction to Algebraic $K$-theory,} GSM 145, AMS 2013.

\bibitem{Yamasaki} A.~Yamasaki, {\it Strong approximation theorem for division algebras over R(X)}, J. Math. Soc. Japan {\bf 49} (1997), no. 3, 455-467.

\end{thebibliography}

\end{document}